\documentclass{amsart}

\usepackage[latin1]{inputenc}
\usepackage[T1]{fontenc}
\usepackage{enumerate}

\newtheorem{theorem}{Theorem}[section]
\newtheorem{lemma}[theorem]{Lemma}
\newtheorem{cor}[theorem]{Corollary}
\newtheorem{prop}[theorem]{Proposition}

\theoremstyle{definition}
\newtheorem{definition}[theorem]{Definition}
\newtheorem{example}[theorem]{Example}
\newtheorem{problem}{Problem}
\newtheorem{question}{Question}

\theoremstyle{remark}
\newtheorem{remark}[theorem]{Remark}

\numberwithin{equation}{section}

\newcommand{\N}{\mathbb N}

\newcommand{\R}{\mathbb R}
\newcommand{\Z}{\mathbb Z}
\newcommand{\C}{\mathbb C}

\title{Existence of common and upper frequently hypercyclic subspaces}
\author[J. B\`{e}s, Q. Menet]{Juan B\`{e}s and Quentin Menet}
\address{Juan B\`{e}s, Department of Mathematics and Statistics,
Bowling Green State University,
Bowling Green, Ohio 43403,
USA}
\email{jbes@bgsu.edu}
\address{Quentin Menet, D\'{e}partement de Math\'{e}matique, Universit\'{e} de Mons, 20 Place du Parc, 7000 Mons, Belgique}
\email{quentin.menet@umons.ac.be}
\thanks{The first author is partially supported by MEC and FEDER Projects MTM2007-64222 and 
MTM2010-14909.  The second author is supported by a grant of FRIA}
\date{August 23rd, 2014}

\allowdisplaybreaks[3]
\begin{document}
\begin{abstract}
We provide criteria for the existence of upper frequently hypercyclic subspaces and for common hypercyclic subspaces, which include the following consequences. There exist frequently hypercyclic operators with upper-frequently hypercyclic subspaces and no frequently hypercyclic subspace. 
On the space of entire functions,  each differentiation operator induced by a non-constant polynomial supports an upper frequently hypercyclic subspace, and the family of its non-zero scalar multiples has a common hypercyclic subspace.  A question of Costakis and Sambarino on the existence of a common hypercyclic subspace for a certain uncountable family of weighted shift operators is also answered.
\end{abstract}
\maketitle
\section{Introduction}
Throughout this paper, $X$ and $Y$ denote separable infinite-dimensional Fr\'{e}chet spaces, and $L(X,Y)$ denotes the space of continuous linear operators from $X$ to $Y$. When $Y=X$, we let $L(X)=L(X,X)$. A sequence $(T_n)$ in $L(X,Y)$ is said to be {\em hypercyclic } provided there is some $x$ in $X$ for which the set $\{ T_nx: \ n\ge 0\}$ is dense in $Y$. Such vector $x$ is called a {\em hypercyclic vector} for $(T_n)$.
 An operator $T$ in $L(X)$ is said to be hypercyclic whenever its sequence of iterates $(T^n)$ is a hypercyclic sequence.

An important question about a hypercyclic operator is whether it supports any infinite-dimensional {\em closed} subspace in which every non-zero vector is hypercyclic. Such a subspace is called a \emph{hypercyclic subspace}. This notion is interesting --see \cite[Chapter 8]{BayartMatheron} or \cite[Chapter 10]{Grosse}-- because some hypercyclic operators do not support hypercyclic subspaces while some other operators do. Indeed, Montes \cite{Montes} showed that no scalar multiple of the backward shift on $\ell^p$ ($1\le p<\infty$)  supports a hypercyclic subspace. On the other hand, the collective work of Bernal and Montes \cite{Bernal}, Petersson \cite{Petersson}, Shkarin \cite{ShkarinD} and Menet \cite{Menet1} shows that each non-scalar convolution operator on the space $H(\mathbb{C})$ of entire functions has a hypercyclic subspace. Convolution operators on $H(\C)$ are those that commute with the operator $D$ of complex differentiation, and are precisely those of the form
\[
\phi (D): H(\C)\to H(\C ), \ f\mapsto \sum_{n\ge 0} a_n D^nf,
\]
where $\phi (z)=\sum_{n\ge 0} a_n z^n \in H(\C )$ is of exponential type \cite{Godefroy}.

\subsection{Two sufficient criteria for the existence of hypercyclic subspaces}  \label{Ss:Hcs}
Several criteria for the existence or the non-existence of hypercyclic subspaces are known; we will be  
particularly interested in the next two theorems. We first recall the following.
\begin{definition}(Le\'{o}n and M\"{u}ller \cite{Leon})\label{def C fre}
 A sequence $(T_n)$ in $L(X,Y)$ satisfies \emph{Condition} (C)  along a given strictly increasing sequence  $(n_k)$ of positive integers provided
\begin{enumerate}
\item[(1)]\ For each $x$ in some dense subset of $X$ we have $T_{n_k}x\to 0$, and
\item[(2)]\  For each continuous seminorm $p$ on $X$, $\bigcup_k T_{n_k}(p^{-1}[0,1))$ is dense in $Y$.
\end{enumerate}
\end{definition}
First formulated on the Banach space setting, Condition (C) ensures by  a standard Baire Category argument that the sequence $(T_n)$ has a residual set of hypercyclic vectors whenever $Y$ is separable \cite{Menet1}.

\begin{theorem}[{\bf Criterion ${\bf M_0}$ }\cite{Menet1}]\label{thm M0}
Let  $(T_n)$ be a sequence in $L(X,Y)$ satisfying Condition~$(C)$ along a sequence $(n_k)$.  Suppose that $X$ supports a continuous norm and that there exists an infinite-dimensional closed subspace $M_0$ of $X$ such that \[T_{n_k}x\rightarrow 0 \ \ \mbox{ for all $x\in M_0$.}\] Then $(T_n)$ has a hypercyclic subspace.
\end{theorem}

\begin{definition}  Let $(M_n)$ be a sequence of infinite-dimensional closed subspaces of $X$ with $M_n\supseteq M_{n+1}$ for all $n$.
A sequence $(T_n)$  in $L(X, Y)$ is said to be {\em equicontinuous along $(M_n)$} provided  for each continuous seminorm $q$ on $Y$ there exists a continuous seminorm $p$ of $X$ so that for each $n\in\mathbb{N}$ 
\[
q(T_nx) \le p(x) \ \ \mbox{ for each $x\in M_n$.}
\]
\end{definition}

\begin{theorem}[ {\bf Criterion ${\bf (M_k)}$} \cite{Menet1}]\label{thm Mk}
Let $(T_n)$ be a sequence in $L(X,Y)$ satisfying Condition $(C)$ along a sequence $(n_k)$. Suppose that $X$ supports a continuous norm and that 
$(T_{n_k})$ is equicontinuous along some non-increasing sequence $(M_k)$ of closed, infinite-dimensional subspaces of $X$.
Then $(T_n)$ has a hypercyclic subspace.
\end{theorem}

The first versions of Criterion $M_0$ and Criterion $(M_k)$ appeared under the stronger assumption of satisfying the Hypercyclicity Criterion instead of Condition $(C)$.
Criterion $M_0$ first appeared in 1996 
for the case of operators on Banach spaces and is due to Montes~\cite{Montes}; this was generalized to operators on Fr\'{e}chet spaces with a continuous norm  by Bonet, Mart\'inez and Peris~\cite{Bonet} and independently by Petersson~\cite{Petersson}, as well as to sequences of operators on Banach spaces by Le\'{o}n and M\"{u}ller~\cite{Leon}. The present version on sequences of operators on Fr\'{e}chet spaces with a continuous norm is due to Menet~\cite{Menet1}. 
Criterion $(M_k)$  first appeared implicitly for the case of operators on Hilbert  and Banach spaces (and where the sequence $(M_k)$ is constant) in the works of  Le\'{o}n and Montes \cite{LeonMontes} and by Gonz\'{a}lez et al \cite{Gonzalez}, respectively; this was extended in 2006 by Le\'on and M\"{u}ller \cite{Leon} to sequences of operators on Banach spaces and more recently by Menet \cite{Menet1} to the present version on Fr\'{e}chet spaces with a continuous norm. Both Criterion $M_0$ and Criterion $(M_k)$ have generalizations to Fr\'{e}chet spaces without continuous norm as well \cite{Menet1bis}.
\vspace{.1in}

\subsection{Two themes about hypercyclic subspaces we consider}

This paper explores a link between Criterion $M_0$ and Criterion $(M_k)$  and its consequences within two themes: subspaces of $\mathcal{U}$-frequently hypercyclic vectors 
and subspaces of common hypercyclic vectors.
The first theme is motivated by recent work of Bonilla and Grosse-Erdmann~\cite{Bonilla2} that initiated the study of {\em frequently hypercyclic subspaces,} that is, of hypercyclic subspaces consisting entirely (but the origin) of frequently hypercyclic vectors. 
The notions of frequently hypercyclic vectors and of $\mathcal{U}$-frequently hypercyclic vectors are due to Bayart and Grivaux~\cite{Bayart04} and to Shkarin~\cite{Shkarin0}, respectively.  For a set $A$ of non-negative integers, the quantities
\[\underline{\text{dens}}(A):=\liminf_N\frac{\#(A\cap [0,N])}{N+1}  \mbox{  and  } \overline{\text{dens}}(A):=\limsup_N\frac{\#(A\cap [0,N])}{N+1} \]  
are respectively called the lower density and the upper density of $A$.

\begin{definition}\label{D:fhc} Let $(T_n)$ be a sequence in $L(X,Y)$. A vector $x\in X$ is said to be {\em frequently hypercyclic} (respectively, $\mathcal{U}$-\emph{frequently hypercyclic}) for $(T_n)$ provided for each non-empty open subset $U$ of $Y$, the set $\{n\ge 0: T_n x\in U\}$ has positive lower density (respectively, positive upper density).
\end{definition}

Bonilla and Grosse-Erdmann \cite{Bonilla3, Bonilla2} showed that each non-scalar convolution operator $\phi (D)$ on $H(\C )$  is frequently hypercyclic and that it has  a frequently hypercyclic subspace whenever $\phi\in H(\mathbb{C})$ is transcendental, and asked whether the derivative operator $D$ has a frequently hypercyclic subspace as well \cite[Problem 2]{Bonilla2}. We will consider here the following related question.

\begin{question} \label{Q:P(D)ufhcs}
Does the derivative operator have a $\mathcal{U}$-frequently hypercyclic subspace on $H(\mathbb{C})$? What about $P(D)$, where
$P$ is a non-constant polynomial?
\end{question}

The next question is motivated by recent work of Bayart and Ruzsa \cite{Bayart3}, who completely characterized frequent hypercyclicity and $\mathcal{U}$-frequent hypercyclicity among unilateral and bilateral weighted shift operators on $\ell^p$ and $\ell^p(\Z )$ $(1\le p<\infty )$, respectively, as well as on  $c_0$ and $c_0(\mathbb{Z})$, respectively. Hence it is natural to ask:

\begin{question} \label{Q:B_wUfhcs}
Which weighted shifts support a $\mathcal{U}$-frequently hypercyclic subspace?
\end{question}

\vspace{.05in}

The second theme we consider is common hypercyclic subspaces. Recall the following.

\begin{definition}
A vector $x\in X$ is called a \emph{common hypercyclic vector} for a given family $\mathcal{F}=\{ (T_{n,\lambda})_{n\ge 0} \}_{\lambda\in \Lambda}$ of sequences in $L(X,Y)$ provided $x$ is a hypercyclic vector for each  $(T_{n,\lambda})_{n\ge 0}$ in $\mathcal{F}$. 
We also say that the common hypercyclic vectors of a given family $\mathcal{F}=\{ T_{\lambda} \}_{\lambda\in \Lambda}$ of operators on $X$ are the common hypercyclic vectors of the family of corresponding sequences of iterates $\{ (T_{\lambda}^n)_{n\ge 0} \}_{\lambda\in \Lambda}$.
A closed, infinite-dimensional subspace consisting entirely (but the origin) of common hypercyclic vectors for $\mathcal{F}$ is called a {\em common hypercyclic subspace} for $\mathcal{F}$. A dense, linear subspace consisting entirely (but the origin) of common hypercyclic vectors for $\mathcal{F}$ is called a {\em common hypercyclic manifold} for $\mathcal{F}$.
\end{definition}
If the family $\mathcal{F}=\{ T_{\lambda} \}_{\lambda\in \Lambda}$ of hypercyclic operators is countable, then it has a residual set of common hypercyclic vectors 
and a common hypercyclic manifold \cite{Grivaux2003}. If $\mathcal{F}$ is uncountable, however, it may fail to have a single common hypercyclic vector, even if each operator $T_\lambda$ has a hypercyclic subspace \cite{Aron}.  The first example of an uncountable family with a common hypercyclic vector was given by Abakumov and Gordon \cite{Abakumov}, who  showed that $\{ \lambda B \}_{\lambda >1}$ has a common hypercyclic vector, where $B$ is the unweighted backward shift on $\ell^2$. The importance of common hypercyclic vectors within linear dynamics is showcased in \cite[Chapter 7]{BayartMatheron} and \cite[Chapter 11]{Grosse}.

In a remarkable paper, Costakis and Sambarino \cite{Costakis2} considered unilateral backward shifts $B_{w_\lambda}$ $(\lambda >1)$ on $\ell^2$ with weight sequence $w_\lambda =(1+\frac{\lambda}{n})_{n\ge 1}$ and
showed that the set of common hypercyclic vectors for the family
\[
\{  B_{w_\lambda }  \}_{\lambda >1 }
\]
is residual on $\ell^2$.  Each such $B_{w_\lambda}$ has a hypercyclic subspace. Hence the following
\begin{question}{ ({\bf Costakis and Sambarino \cite[Problem 3]{Costakis2}})}  \label{Q:Bwchcs}
Does  the family $\{ B_{w_\lambda} \}_{\lambda >1}$ support a common hypercyclic subspace?
\end{question}

We already mentioned that any given non-scalar convolution operator $\phi(D)$ on $H(\C )$ has a hypercyclic subspace. But the family of its (non-zero) scalar multiples
\[
\{ \lambda\, \phi(D) \}_{|\lambda |>0}
\]
has a residual set of common hypercyclic vectors, as established by Costakis and Mavroudis~\cite{Costakis} for the case $\phi$ is a non-constant polynomial, by Bernal~\cite{Bernal2009} for the case $\phi$ is of growth order no larger than $\frac{1}{2}$, and by Shkarin \cite{Shkarin} for the general transcendental case. Hence it is natural to ask.

\begin{question} \label{Q:P(D)common}
Does the family of non-zero scalar multiples of a non-scalar convolution operator on $H(\C )$ support a common hypercyclic subspace?
\end{question}

We mention that Bayart~\cite{Bayart} extended  Criterion~$M_0$ to one that shows the existence of common hypercyclic subspaces (see Theorem~\ref{CritM0}), but this extension has proven difficult to apply in many natural examples.

\subsection{Main results and outline of the paper}
Section~\ref{Gene Mk} is devoted to developing the main tools (Theorem~\ref{NiceMn}) for establishing in the subsequent sections $(M_k)$-type criteria for the existence of $\mathcal{U}$-frequently hypercyclic subspaces or of common hypercyclic subspaces. 

Section~\ref{SecU} is devoted to $\mathcal{U}$-frequently hypercyclic subspaces. We show for instance that any operator on a complex Banach space that satisfies the Frequent Hypercyclicity Criterion (Theorem~\ref{T:FHC}) {\em has a $\mathcal{U}$-frequently hypercyclic subspace if and only if $T$ has a hypercyclic subspace} (Theorem~\ref{T:FHC2ufhcs}). We then answer Question~\ref{Q:P(D)ufhcs}  in the affirmative. Indeed, we show that  {\em for any $N\ge 1$ and any $a_0(z)$,\dots, $a_{N-1}(z)\in H(\C )$ and $0\ne a_N\in\C$ the linear differential operator
\[
L=a_N D^N+ a_{N-1}(z) D^{N-1}+\cdots +a_0(z) I
\]
has a $\mathcal{U}$-frequently hypercyclic subspace on $H(\C )$} (Corollary~\ref{C:ufhsNonConstant}). Concerning Question~\ref{Q:B_wUfhcs}, we show that {\em each $\mathcal{U}$-frequently hypercyclic unilateral weighted backward shift with a hypercyclic subspace on $\ell^p$ has a $\mathcal{U}$-frequently hypercyclic subspace} (Corollary~\ref{caraclp}).
As a consequence, we derive the existence of {\em a frequently hypercyclic operator that has a $\mathcal{U}$-frequently hypercyclic subspace but no frequently hypercyclic subspace} (Corollary~\ref{C:ufhcsNofhcs}). We complement our answer to Question~\ref{Q:B_wUfhcs} by showing that each $\mathcal{U}$-frequently hypercyclic bilateral weighted backward shift with a hypercyclic subspace on $\ell^p(\Z)$ has even a frequently hypercyclic subspace (Theorem~\ref{caraclpbi}).

We obtain the above results after establishing the following criterion for supporting $\mathcal{U}$-frequently hypercyclic subspaces (cf.Theorem~\ref{UM0} and Theorem~\ref{UMk}): {\em If the Fr\'{e}chet space $X$ supports a continuous norm then any $T\in L(X)$ satisfying the Frequent Hypercyclicity Criterion has a $\mathcal{U}$-frequently hypercyclic subspace provided there exists a strictly increasing sequence of positive integers $(n_k)$ with positive upper density satisfying either of the two conditions:
\begin{enumerate}
\item[(a)]\ $(T^{n_k})$ converges pointwise to zero on some closed and infinite-dimensional subspace $M_0$ of $X$, or

\item[(b)]\ $(T^{n_k})$ is equicontinuous along some non-increasing sequence $(M_k)$ of closed infinite-dimensional subspaces of $X$.
\end{enumerate}}

Section~\ref{Sec common} is devoted to common hypercyclic subspaces. We provide with Theorem~\ref{CritM0} a constructive proof of a Criterion $M_0$ due to Bayart \cite{Bayart}  for the existence of common hypercyclic subspaces but on Fr\'{e}chet spaces that support a continuous norm, and use it to establish a corresponding Criterion $(M_k)$ which is simpler to apply (Theorem~\ref{Mk com}). Indeed, we use the latter to answer Question~\ref{Q:Bwchcs}
in the affirmative (Corollary~\ref{Ex:Costa}) and to partially anwser Question~\ref{Q:P(D)common}: {\em for any $N\ge 1$ and any  $a_0(z)$,\dots , $a_{N-1}(z)\in H(\C )$, the family
\[
\{ \lambda L \}_{0\ne \lambda\in\C}
\]
of non-zero scalar multiples of the linear differential operator \[ L=D^n+a_{N-1}(z)D^{N-1}+\cdots +a_0(z)I\] has a common hypercyclic subspace on $H(\C )$} (Corollary~\ref{C:NonConstantCommon}). Finally, we seek in the last subsection spectral characterizations for the existence of common hypercyclic subspaces for certain families of operators, complementing the  spectral characterization by Gonz\'{a}lez, Le\'{o}n and Montes \cite{Gonzalez} for the case of a single operator. In particular, we show that 
{\em  for a complex Banach space $X$ and  $0< a<b\le \infty$, a family of non-zero scalar mutiples of a given $T\in L(X)$
\[
\{ \lambda T \}_{a<\lambda <b}
\]
that satisfies the Common Hypercyclicity Criterion has a common hypercyclic subspace if and only if the essential spectrum of $T$ contains an element of modulus no larger than $\frac{1}{b}$, with the convention that $\frac{1}{b}=0$ when $b=\infty$} (Corollary~\ref{C:a<lambda<b} and Corollary~\ref{C:Spectral}).

We finish this introduction with a brief subsection on the Common Hypercyclicity Criterion, which we use in Section~\ref{Sec common}.

\subsection{The Common Hypercyclicity Criterion} \label{Ss:CHC}
Several criteria have been used to show the existence of common hypercyclic vectors thanks to the works of Costakis and Sambarino \cite{Costakis2}, Bayart and Matheron \cite{Bayart2}, and others  \cite{Bayart2004, BayartGrivaux, Bernal2009,  Costakis, Gallardo, Grosse}. In this paper we use the following version of the Common Hypercyclicity Criterion based on \cite[Remark 11.10]{Grosse}.

\begin{definition}[{\bf Common Hypercyclicity Criterion}]\label{D:CHC}
Let $\Lambda \subseteq \mathbb{R}$ be an open interval. We say that a family $\{ (T_{n,\lambda})_{n\ge 0} \}_{\lambda\in \Lambda}$ of sequences in $L(X,Y)$ satisfies the \emph{Common Hypercyclicity Criterion} (CHC) provided
 for each $(n, x)\in \mathbb{Z}^+\times X$ the vector $T_{n,\lambda}x$ depends continuously on $\lambda\in \Lambda$ and provided  for each compact subset $K$ of $\Lambda$ there exist dense subsets $X_0$ and $Y_0$ of $X$ and $Y$, respectively, and mappings 
\[
S_{n, \lambda}: Y_0\to X \ \ \ (\lambda\in K, \ n=0,1,\dots )
\]
so that for each $y_0\in Y_0$, each $x_0\in X_0$, we have
\begin{enumerate}
\item[(1)]\  $\sum_{k=0}^m T_{m, \lambda} S_{m-k, \mu_k} y_0$ converges unconditionally and uniformly on $\lambda \ge \mu_0\ge \mu_1\ge\dots \ge \mu_m$ in $K$ and $m\ge 0$.

\item[(2)]\  $\sum_{k=0}^\infty T_{m, \lambda} S_{m+k, \mu_k} y_0$ converges unconditionally and uniformly on $\lambda \le \mu_0\le \mu_1\le\dots $ in $K$ and $m\ge 0$.

\item[(3)]\  For each $\varepsilon>0$ and each continuous seminorm $q$ on $Y$, there exists a sequence $(\delta_{k})$ of positive real numbers such that $\sum_{k=0}^{\infty}\delta_k=\infty$ and  for each $k\ge 0$ and $\lambda, \mu\in K$ we have
\[
0\le \mu-\lambda < \delta_{k}  \ \   \Rightarrow \ \ \  q( T_{k, \lambda} S_{k,\mu} y_0 - y_0) < \epsilon.
\]

\item[(4)]\   $T_{k, \lambda} x_0 \xrightarrow[k\to\infty]{} 0$\ uniformly on $\lambda \in K$.

\item[(5)]\ $\sum_{k=0}^\infty \ S_{k, \mu_k} y_0$ converges unconditionally and uniformly on $\mu_0\le \mu_1\le \dots$ in $K$.   
\end{enumerate}
\end{definition} 
\begin{theorem} 
 \label{T:CHC}
Let $\{ (T_{n,\lambda})_n \}_{\lambda\in \Lambda}$ be a family of sequences in $L(X, Y)$ satisfying (CHC).  Then the set of vectors in X that are hypercyclic for every sequence $(T_{n,\lambda})_n$  is residual in $X$.
\end{theorem}

Theorem~\ref{T:CHC} (follows from the same arguments and) slightly generalizes the Common Hypercyclicity Criterion given in \cite[Theorem~11.9, Remark 11.10(d)]{Grosse}.
\begin{remark} \label{R:CHC}\
\begin{enumerate}
\item[(a)]\ If $\{ T_{\lambda} \}_{\lambda \in \Lambda}$ is a family of operators satisfying (the assumptions of) the Common Hypercyclicity Criterion of Costakis and Sambarino~\cite{Costakis2}, then the family $\{ (T_{\lambda}^{n+1})_{n\ge 0} \}_{\lambda\in\Lambda}$ satisfies $(CHC)$ as defined above in Definition~\ref{D:CHC}. 

\item[(b)]\ Given a scalar $\lambda_0\ge 0$ and $T\in L(X)$, the family $\{ (\lambda^n T^n)_{n\ge 1} \}_{\lambda >\lambda_0}$ satisfies $(CHC)$ if there exists a dense subset $A$ of $X$ that is contained in $\cup_{n\ge 1} \mbox{Ker}(T^n)$ and maps $S_n:A\to X$ so that for each $x\in A$ we have that $(i)$  $T^nS_nx=x$, that $(ii)$ $T^mS_{m+n}x=S_nx$ for each $m,n\ge 0$, and that  $(iii)$ $\{\frac{1}{\lambda^n} S_nx \}_{n\ge 1}$ is bounded in $X$ for each $\lambda >\lambda_0$, see \cite[Section 11.2]{Grosse}.
\end{enumerate}
\end{remark}

\section{Main tool for the generalization of Criterion $(M_k)$}\label{Gene Mk}

In this section we develop a link between Criterion $M_0$ and Criterion $(M_k)$. Notice that Criterion $M_0$ immediately gives Criterion $(M_k)$, by the Banach-Steinhaus theorem. Conversely, if a sequence of operators $(T_n)$ satisfies Criterion $(M_k)$ along a given sequence of integers $(n_k)$, then it also satisfies Criterion $M_0$ for some subsequence $(m_k)$ of $(n_k)$ \cite{Menet1}. The section's main result, Theorem~\ref{NiceMn},  allows us to get this subsequence $(m_k)$ to inherit special properties of the original sequence $(n_k)$ (see Lemma~\ref{L:52} and Proposition~\ref{prophered}) which is a key ingredient for establishing  $(M_k)$-type criteria for the existence of $\mathcal{U}$-frequently hypercyclic subspaces or common hypercyclic subspaces from the corresponding $M_0$-type criteria.
We need the following two definitions.

\begin{definition}
We say that $(\Lambda_n)_{n\ge 1}$ is a \emph{chain} of a given set $\Lambda$ if 
the sequence $(\Lambda_n)_{n\ge 1}$ is non-decreasing and $\bigcup_{n\ge 1}\Lambda_n=\Lambda$.
\end{definition}


\begin{definition}  Let $(M_n)$ be a sequence of infinite-dimensional closed subspaces of $X$ with $M_n\supseteq M_{n+1}$ for all $n$.
A family  $\{ (T_{n, \lambda} ) \}_{\lambda\in \Lambda}$ of sequences in $L(X, Y)$ is said to be {\em uniformly equicontinuous along $(M_n)$} provided  for each continuous seminorm $q$ on $Y$ there exists a continuous seminorm $p$ of $X$ so that for each $n\in\mathbb{N}$ 
\[
\mbox{sup}_{\lambda\in\Lambda} q(T_{n, \lambda}x) \le p(x) \ \ \mbox{ for each $x\in M_n$.}
\]
\end{definition}

\begin{theorem}\label{NiceMn} Let  $X$ be an infinite-dimensional Fr\'{e}chet space with continuous norm, let $Y$ be a separable Fr\'echet space and let $\Lambda$ be a set.
Let $\{ (T_{k, \lambda})_{k\ge 1} \}_{\lambda\in\Lambda}$ be a family of sequences of operators in $L(X,Y)$.  Suppose that there exist chains   $(\Lambda^j_n)_{n\ge 1}$ of $\Lambda$ $(j=0,1,2)$  satisfying:
\begin{enumerate}
\item[(i)]\  for each $n\in\mathbb{N}$ and each $k\in \N$, the family
$
\{  T_{k, \lambda} \}_{\lambda\in \Lambda_n^0} \ \mbox{ is equicontinuous;}
$
\item[(ii)]\  for each $n\in\mathbb{N}$, there exists a dense subset $X_{n,0}$ of  $X$ such that for any $x\in X_{n,0}$,
\[
T_{k,\lambda}x\xrightarrow[k\to\infty]{} 0 \ \ \ \mbox{uniformly on $\lambda\in \Lambda_n^1$;}
\]
\item[(iii)]\ there exists a non-increasing sequence of infinite-dimensional closed subspaces $(M_k)$ of $X$ such that for each $n\ge 1$,  the family of sequences
\[
\{   (T_{k,\lambda})_{k\ge 1}\}_{\lambda\in \Lambda_n^2}
\]
is uniformly equicontinuous along $(M_k)$.
\end{enumerate}
Then for any map $\phi:\mathbb{N}\to \mathbb{N}$ there exist an increasing sequence of integers $(k_s)_{s\ge 1}$ and an infinite-dimensional closed subspace $M_0$ of $X$ such that for any $(x,\lambda)\in M_0\times \Lambda$,
\[T_{k,\lambda}x\xrightarrow[k\to \infty]{k\in I} 0,\]
where $I=\bigcup_{s\ge 1}[k_s,k_s+\phi(k_s)]$. 
\end{theorem}

The choice of the function $\phi:\N\to \N$ in Theorem~\ref{NiceMn} is the main tool to generalize Criterion~$(M_k)$ to $\mathcal{U}$-frequently hypercyclic subspaces (Section~\ref{SecU}) and to common hypercyclic subspaces (Section~\ref{Sec common}).
The proof of Theorem~\ref{NiceMn} relies on the notion of a basic sequence, which has been commonly used for the construction of closed, infinite-dimensional subspaces whose vectors are to satisfy special properties. We refer to \cite{Menet1, Petersson} for more details about the construction of basic sequences in Fr\'{e}chet spaces with a continuous norm.

\begin{definition}
A sequence $(u_k)_{k\ge 1}$ in a Fr\'{e}chet space is called \emph{basic} if for every $x\in \overline{\text{span}}\{u_k:k\ge 1\}$, there exists a unique sequence $(a_k)_{k\ge 1}$ in $\mathbb{K}$ ($\mathbb{K}=\mathbb{R}$ or $\mathbb{C}$) such that $x=\sum_{k=1}^{\infty}a_ku_k$.
\end{definition}

\begin{proof}[Proof of Theorem \ref{NiceMn}]\  Let $(p_j)$ be an increasing sequence of norms inducing the topology of $X$ and let $(q_j)$ be an increasing sequence of seminorms inducing the topology of $Y$. 
We consider $\Lambda_n=\Lambda^0_n\cap \Lambda^1_n\cap \Lambda^2_n$ and a basic sequence $(u_i)_{i\ge1}$ in $X$ such that $u_i\in M_i$ and  $p_1(u_i)=1$ for each $i\ge 1$,  and so that for every $j\ge 1$, the sequence $(u_i)_{i\ge j}$ is basic in $X_j:=(X,p_j)$ with basic constant less than $2$.
By $(i)$, for each $n, k, j\ge 1$ , there exists $K_{n,k,j}>0$ and $m^0(n,k,j)\in\mathbb{N}$ such that for any $x\in X$,
\begin{equation}
\sup_{\lambda\in \Lambda_n}q_j(T_{k,\lambda}x)\le K_{n,k,j} \ p_{m^0(n,k,j)}(x).
\label{eqi)}
\end{equation}

By $(iii)$, for each $n,j\ge 1$ there exist  $C_{n,j}>0$ and  $m(n,j)\in\mathbb{N}$
such that for any $k\ge 1$ and  any $x\in M_k$,
\begin{equation}
\sup_{\lambda\in \Lambda_n}q_j(T_{k,\lambda}x)\le C_{n,j} p_{m(n,j)}(x).
\label{eqiii)}
\end{equation}

Since each set $X_{i,0}$ is dense, we can construct a family $(x_{i,l})_{i,l\ge 1}\subset X$ and a sequence of integers $(k_l)_{l\ge 0}$ with $k_0=1$ such that for any $l\ge 1$,
\begin{enumerate}
\item for any $i\ge 1$, any ${n,k,j\le i}$,
\begin{equation}
p_{m^0(n,k,j)}(x_{i,l})<\frac{1}{2^{i+l}K_{n,k,j}} \quad \text{and} \quad p_i(x_{i,l})<\frac{1}{2^{i+l+2}};
\label{3prop f}
\end{equation}
\item for any $i,k\le k_{l-1}+\phi(k_{l-1})$,
\begin{equation}
\sup_{\lambda\in \Lambda_{k_{l-1}+\phi(k_{l-1})}}q_{k_{l-1}+\phi(k_{l-1})}(T_{k,\lambda}x_{i,l})<\frac{1}{2^{i+l}};
\label{3prop f2}
\end{equation}
\item for any $i\ge 1$, $u_i+\sum_{l'=1}^{l}x_{i,l'}\in X_{k_{l-1}+\phi(k_{l-1}),0}$;
\item $k_{l}> k_{l-1}+\phi(k_{l-1})$;
\item for any $i\le k_{l-1}+\phi(k_{l-1})$, any $k\in [k_l,k_l+\phi(k_l)]$,
\begin{equation}
\sup_{\lambda\in \Lambda_{k_{l-1}+\phi(k_{l-1})}}q_{k_{l-1}+\phi(k_{l-1})}\Big(T_{k,\lambda} \Big(u_i+\sum_{l'=1}^{l}x_{i,l'}\Big)\Big)<\frac{1}{2^{l+i}}.
\label{3prop k}
\end{equation}
\end{enumerate}
Indeed, in order to satisfy (1) and (2), it suffices to choose $x_{i,l}$ sufficiently small thanks to \eqref{eqi)} and since we choose $x_{i,l}$ such that $u_i+\sum_{l'=1}^{l}x_{i,l'}\in X_{k_{l-1}+\phi(k_{l-1}),0}$, we know that for any $i\ge 1$,
$T_{k,\lambda}\Big(u_i+\sum_{l'=1}^{l}x_{i,l'}\Big)$ tends uniformly to $0$ on $\lambda\in \Lambda_{k_{l-1}+\phi(k_{l-1})}$. We can thus find $k_l$ sufficiently big such that (4) and (5) are satisfied.

For any $n\ge 1$, we let $x_n=u_n+\sum_{l=1}^{\infty}x_{n,l}$. We deduce from \eqref{3prop f} that $(x_{k_l+\phi(k_l)})_{l\ge 1}$ is a basic sequence equivalent to the sequence $(u_{k_l+\phi(k_l)})_{l\ge 1}$. Let $M_0$ be the closed linear span of $(x_{k_l+\phi(k_l)})_{l\ge 1}$ in $X$ and $x$ a vector in $M_0$. We know that we have $x=\sum_{s=1}^{\infty} a_s x_{k_s+\phi(k_s)}$ where the sequence $(a_s)_{s\ge 1}$ is bounded by some constant $K$. 

Let $n,j\ge 1$ and $l\ge 2$ with $n,j\le k_{l-1}+\phi(k_{l-1})$. Since 
\begin{equation*}
{\sum_{s=l}^{\infty}a_s u_{k_s+\phi(k_s)}\in M_{k_l+\phi(k_l)}},
\end{equation*}
 we deduce from \eqref{eqi)}, \eqref{eqiii)}, \eqref{3prop f}, \eqref{3prop f2} and \eqref{3prop k} that for any $\lambda\in \Lambda_n$,  any $k\in [k_l,k_l+\phi(k_l)]$,
\begin{align*}
q_j(T_{{k},\lambda}x)&\le \sum_{s=1}^{l-1}|a_s| q_j(T_{{k},\lambda}x_{k_s+\phi(k_s)})
+ \sum_{s=l}^{\infty} |a_s| q_j(T_{{k},\lambda}(x_{k_s+\phi(k_s)}-u_{k_s+\phi(k_s)}))\\
&\quad+ q_j\Big(T_{{k},\lambda}\Big(\sum_{s=l}^{\infty}a_su_{k_s+\phi(k_s)}\Big)\Big)\\
&\le \sum_{s=1}^{l-1}|a_s| q_{k_{l-1}+\phi(k_{l-1})}\Big(T_{{k},\lambda}\Big(u_{k_s+\phi(k_s)}+\sum_{l'=1}^{l}x_{k_s+\phi(k_s),l'}\Big)\Big)\\
&\quad+ \sum_{s=1}^{l-1}\sum_{l'=l+1}^{\infty}|a_s| q_{k_{l'-1}+\phi(k_{l'-1})}(T_{{k},\lambda}x_{k_s+\phi(k_s),l'})\\
&\quad+ \sum_{s=l}^{\infty} \sum_{l'=1}^{\infty}|a_s| q_j(T_{{k},\lambda}x_{k_s+\phi(k_s),l'})+ q_j\Big(T_{{k},\lambda}\Big(\sum_{s=l}^{\infty}a_su_{k_s+\phi(k_s)}\Big)\Big)\\
&\le \sum_{s=1}^{l-1} \frac{K}{2^{l+k_s+\phi(k_s)}}+\sum_{s=1}^{l-1} \sum_{l'=l+1}^{\infty}\frac{K}{2^{k_s+\phi(k_s)+l'}}\\
&\quad +K \sum_{s=l}^{\infty} \sum_{l'=1}^{\infty} K_{n,k,j} p_{m^0(n,k,j)}(x_{k_s+\phi(k_s),l'}) + C_{n,j}p_{m(n,j)}\Big(\sum_{s=l}^{\infty}a_s u_{k_s+\phi(k_s)}\Big)\\
&\le \frac{lK}{2^{l-1}}
+ K\sum_{s=l}^{\infty}   \frac{K_{n,k,j}}{2^{k_s+\phi(k_s)}K_{n,k,j}}+ C_{n,j}p_{m(n,j)}\Big(\sum_{s=l}^{\infty}a_s u_{k_s+\phi(k_s)}\Big)\\
&\le \frac{lK}{2^{l-1}}
+  K\sum_{s=l}^{\infty} \frac{1}{2^{k_s+\phi(k_s)}}  + C_{n,j}p_{m(n,j)}\Big(\sum_{s=l}^{\infty}a_s u_{k_s+\phi(k_s)}\Big)
 \underset{l\rightarrow \infty}{\longrightarrow} 0.
\end{align*}
\end{proof}

 We finish this section by stating two particular cases of Theorem~\ref{NiceMn}. In Corollary~\ref{NiceMnBan} below we consider the case when   $\{ T_{\lambda}\}_{\lambda\in \Lambda}$ is a family of operators on a Banach space $X$. We use Corollary~\ref{NiceMnBan} in Section~\ref{Sec:Charac} to study the existence of common hypercyclic subspaces for collections of scalar multiples of a fixed operator on a complex Banach space.

\begin{cor}\label{NiceMnBan}
Let $(X,\|\cdot\|)$ be a separable infinite-dimensional Banach space, let $\{ T_{\lambda}\}_{\lambda\in \Lambda}$ be a family of operators in  $L(X)$, and let  $(n_k)_{k\ge 1}$ be a strictly increasing sequence of  positive integers.
Suppose there exist chains  $(\Lambda^1_n)$ and $(\Lambda^2_n)$  of $\Lambda$ satisfying:
\begin{enumerate}
\item[(i)]\  for any $n\ge 1$, there exists a dense subset $X_{n,0}$ of $X$ such that for any $x\in X_{n,0}$, \[\sup_{\lambda\in \Lambda^1_n} \|T^{n_k}_{\lambda}x\|\xrightarrow[k\to \infty]{} 0;\]
\item[(ii)]\ there exists  a non-increasing sequence $(M_k)$ of infinite-dimensional closed subspaces of $X$ such that for any $n\in\mathbb{N}$, 
\[\sup_{(\lambda, k)\in \Lambda^2_n\times \mathbb{N}}\|T^{n_k}_{\lambda|M_k}\|<\infty.\]
\end{enumerate}
Then for any $\phi:\mathbb{N}\to \mathbb{N}$, there exist an increasing sequence of integers $(k_s)_{s\ge 1}$ and an infinite-dimensional closed subspace $M_0$ of $X$ such that for any $x\in M_0$ and any $\lambda\in \Lambda$,
\[T^{n_k}_{\lambda}x\xrightarrow[k\to \infty]{k\in I} 0\]
where  $I=\bigcup_{s\ge 1}[k_s,k_s+\phi(k_s)]$.  
\end{cor}
\begin{proof}
In view of Theorem~\ref{NiceMn}, it suffices to prove the existence of a chain $(\Lambda^0_n)$ of $\Lambda$ such that for any $n,k\ge 1$, there exists $K_{n,k}$ such that for any $x\in X$,
\[\sup_{\lambda\in \Lambda^0_n}\|T^{n_k}_{\lambda}x\|\le K_{n,k}\|x\|.\]
For each $n\in\mathbb{N}$, let $\Lambda^0_n=\{\lambda\in \Lambda:\|T_\lambda\|\le n\}$. So $(\Lambda^0_n)$ is a chain of $\Lambda$ and for any $n,k\in\mathbb{N}$ and $x\in X$, we have
\[\sup_{\lambda\in \Lambda^0_n}\|T^{n_k}_{\lambda}x\|
\le \sup_{\lambda\in \Lambda^0_n}\|T^{n_k}_{\lambda}\| \|x\|
\le \sup_{\lambda\in \Lambda^0_n}\|T_{\lambda}\|^{n_k} \|x\|
\le n^{n_k}\|x\|.\]
We conclude that the chain $(\Lambda^0_n)$ satisfied the desired inequalities for $K_{n,k}=n^{n_k}$.
\end{proof}

 Finally, for the case of a single sequence of operators $(T_n)$ in $L(X,Y)$,  Theorem~\ref{NiceMn} gives the following.

\begin{cor}\label{NiceMnone}
Let $X$ be an infinite-dimensional Fr\'{e}chet space with a continuous norm, let $Y$  be a separable Fr\'{e}chet space  and let $(T_n)$ be a sequence in $L(X,Y)$. Suppose there exists a strictly increasing sequence of integers  $(n_k)$ satisfying
\begin{enumerate}
\item[(a)]\  $T_{n_k}x\underset{k\to \infty}{\to} 0$ for each $x$ in some dense subset $X_0$ of $X$, and

\item[(b)]\  $(T_{n_k})_{k\ge 1}$ is equicontinuous along some non-increasing sequence $(M_k)$ of infinite-dimensional closed subspaces of $X$.
\end{enumerate}
Then for any $\phi:\mathbb{N}\to \mathbb{N}$ there exist an increasing sequence of integers $(k_s)_{s\ge 1}$ and an infinite-dimensional closed subspace $M_0$ such that if  $I=\bigcup_{s\ge 1}[k_s,k_s+\phi(k_s)]$, then for any $x\in M_0$,
\[T_{n_k}x\xrightarrow[k\to \infty]{k\in I} 0.\]
\end{cor}
We use Corollary~\ref{NiceMnone} in Section~\ref{SecU} to obtain a version of Criterion~$(M_k)$ for the existence of $\mathcal{U}$-frequently hypercyclic subspaces. 

\section{Existence of $\mathcal{U}$-frequently hypercyclic subspaces}\label{SecU}

The existence of frequently hypercyclic subspaces was first investigated by Bonilla and Grosse-Erdmann~\cite{Bonilla2}, who generalized Criterion~$M_0$ to one for obtaining frequently hypercyclic subspaces, which  has in turn  been applied to convolution operators and weighted composition operators, see \cite{Bonilla2}, \cite{Bes}, \cite{Bes2}.
We first recall the Frequent Hypercyclicity Criterion due to Bayart and Grivaux~\cite{Bayart04}; the version we use here is due to Bonilla and Grosse-Erdmann~\cite{Bonilla3}.
\begin{theorem}[{\bf Frequent Universality Criterion}]  \label{T:FHC}
Let $X$ be a Fr\'{e}chet space, let $Y$ be a separable Fr\'{e}chet space and let $(T_n)$ be a sequence in $L(X,Y)$. If there exist a dense subset $Y_0\subset Y$ and $S_k:Y_0\rightarrow X$, $k\ge 0$ such that for each $y\in Y_0$,
\begin{enumerate}[\upshape 1.]
\item $\sum_{n=0}^{\infty}S_ny$ converges unconditionally in $X$,
\item $\sum_{n=1}^{k}T_{k}S_{k-n}y$ converges unconditionally in $Y$, uniformly in $k\ge 0$,
\item $\sum_{n=1}^{\infty}T_{k}S_{k+n}y$ converges unconditionally in $Y$, uniformly in $k\ge 0$,
\item $T_nS_ny\rightarrow y$,
\end{enumerate}
then $(T_n)$ is frequently hypercyclic.
\end{theorem}

We can now state (a slightly more general formulation of)  the Bonilla and Grosse-Erdmann's criterion for the existence of frequently hypercyclic subspaces.  
\begin{theorem}  {\bf (Bonilla and Grosse-Erdmann)} \label{T:FHCS}
Let $X$ be a Fr\'{e}chet space with a continuous norm and let $Y$ be a separable Fr\'{e}chet space. 
Let $(T_n)$ be a sequence in $L(X, Y)$ satisfying the Frequent Universality Criterion.  Suppose that
\[
T_nx\xrightarrow[n\to\infty]{} 0
\]
for each vector $x$ in some closed, infinite-dimensional subspace $M_0$ of $X$.  Then $(T_n)$ has a frequently hypercyclic subspace.
\end{theorem}
 Under the assumption that $(T_n)$ satisfies the Frequent Universality Criterion, their proof of the above theorem (see \cite[Theorem 3]{Bonilla2}) can then be adapted to $\mathcal{U}$-frequently hypercyclic subspaces as follows:

\begin{theorem}[{\bf Criterion~$M_0$ for $\mathcal{U}$-Frequently Hypercyclic Subspaces}]\label{UM0}
Let $X$ be a Fr\'{e}chet space with a continuous norm and let $Y$ be a separable Fr\'{e}chet space. Let $(T_n)$ be a sequence in $L(X,Y)$ satisfying the Frequent Universality Criterion and for which there exist an infinite-dimensional closed subspace $M_0$ and a strictly increasing sequence $(n_k)_{k\ge 1}$ of positive upper density such that for any $x\in M_0$,
\[T_{n_k}x\xrightarrow[k\to \infty]{} 0.\]
Then $(T_n)$ has a $\mathcal{U}$-frequently hypercyclic subspace.
\end{theorem} 
\begin{proof}
By applying Theorem~\ref{T:FHCS} to the sequence of operators $(T_{n_k})_{k\ge 1}$, we obtain an infinite-dimensional closed subspace $M$ such that for any non-zero vector $x\in M$ and any non-empty open set $U$ in $Y$ the return set $\{k\ge 1:T_{n_k} x\in U\}$ is a set of positive lower density. Since $(n_k)_{k\ge 1}$ is a set of positive upper density, the set $(n_k)_{k\in A}$ is a set of positive upper density for any set $A$ of positive lower density. Indeed, we have
\[\frac{\#(\{n_k:k\in A\}\cap[0,N])}{N+1}=\frac{\#(A\cap[1,j])}{j} \frac{\#(\{n_k:k\in \N\}\cap[0,N])}{N+1},\]
where $j=\# (\{n_k:k\in \N\}\cap[0,N])$. We conclude that $M$ is a $\mathcal{U}$-frequently hypercyclic subspace.
\end{proof}

We next derive the following $(M_k)$-criterion for the existence of  $\mathcal{U}$-frequently hypercyclic subspaces, which is simpler to apply than Theorem~\ref{UM0}. 

\begin{theorem}[{\bf Criterion~${\bf (M_k)}$ for $\mathcal{U}$-Frequently Hypercyclic Subspaces}]\label{UMk}
Let $X$ be an infinite-dimensional Fr\'{e}chet space with a continuous norm, let $Y$ be  a separable Fr\'{e}chet space and let $(T_n)$ be a sequence in $L(X,Y)$ satisfying the Frequent Universality Criterion. 
Suppose that there exists a strictly increasing sequence $(n_k)$ of positive upper density so that
 \begin{enumerate}
 \item \  $T_{n_k}x\underset{k\to\infty}\to 0$ for each $x$ in some dense subset  $X_0$ of $X$, and 
 \item \ $(T_{n_k})_{k\ge 1}$ is equicontinuous along some non-increasing sequence $(M_k)$ of infinite-dimensional closed subspaces of $X$.
 \end{enumerate}
Then $(T_n)$ has a $\mathcal{U}$-frequently hypercyclic subspace.
\end{theorem}

We will prove Theorem~\ref{UMk} with Theorem~\ref{UM0} and by applying Corollary~\ref{NiceMnone} with a suitable map $\phi$ whose existence is given by the following lemma. 

\begin{lemma}  \label{L:52} Let $(n_k)$ be a  strictly increasing sequence  in $\mathbb{N}$  
with positive upper density.
Then there exists a map $\phi:\mathbb{N}\to \mathbb{N}$ so  that for every strictly increasing sequence $(k_s)$ in $\mathbb{N}$ we have
\[
\overline{\text{\emph{dens}}}(\{ n_k:\ k\in \cup_{s\ge 1} [k_s, k_s+\phi (k_s)] \}) = \overline{\text{\emph{dens}}}(n_k).
\]
\end{lemma}
\begin{proof} Let  $\delta = \overline{\text{dens}}(n_k)$, and choose $\phi:\mathbb{N}\to \mathbb{N}$ such that
for any $k\ge 1$,
 \[\frac{\phi(k)+1}{n_{k+\phi(k)}}\ge\delta-\frac{\delta}{k}.\]
Such a map $\phi$ exists because for any $k\ge 1$,
\[\delta=\limsup_{N}\frac{\#\{j: n_j\in [n_k,N]\}}{N}=\limsup_{l}\frac{\#\{j: n_j\in [n_k,n_l]\}}{n_l}=\limsup_{l}\frac{l+1-k}{n_l}\]
and thus for any $k\ge 1$ there exists $\phi(k)$ such that
 \[\frac{\phi(k)+1}{n_{k+\phi(k)}}\ge\delta-\frac{\delta}{k}.\]
The upper density of the subsequence $(m_l)=\{n_k: k\in\bigcup_{s\ge 1}[k_s,k_s+\phi(k_s)]\}$ 
of $(n_k)$ equals~{$\delta$.} Indeed, $\overline{\text{dens}}(m_l)\le \overline{\text{dens}}(n_k)=\delta$
since $(m_l)$ is a subsequence of $(n_k)$,  and on the other hand
\begin{align*}
\overline{\text{dens}}(m_l)&=\limsup_{N}\frac{\#\{j: m_j\in [0,N]\}}{N}\\
&\ge \limsup_{s}\frac{\#\{j: m_j\in [n_{k_s},n_{k_s+\phi(k_s)}]\}}{n_{k_s+\phi(k_s)}}\\
&= \limsup_{s}\frac{\phi(k_s)+1}{n_{k_s+\phi(k_s)}}\ge \limsup_{s}\Big(\delta-\frac{\delta}{k_s}\Big)=\delta.
\end{align*}
So Lemma~\ref{L:52} follows.
\end{proof}

\begin{proof}[Proof of Theorem~\ref{UMk}]\   Let $\phi$ as in Lemma~\ref{L:52}. By
 Corollary~\ref{NiceMnone}, there exists a closed infinite-dimensional subspace $M_0$ of $X$ and a strictly increasing sequence $(k_s)$ in $\mathbb{N}$ such that  for any $x\in M_0$,
 \[
 T_{n_k}x\xrightarrow[k\to\infty]{k\in I} 0
 \]
where $I=\cup_{s\ge 0} [k_s, k_s+\phi (k_s)]$.  Since $(n_k)$ is a sequence of positive upper density, the subsequence $(m_l)=\{n_k:k\in I\}$ has positive upper density (Lemma~\ref{L:52}) and  the conclusion now follows from Theorem~\ref{UM0}.
\end{proof}

If we consider an operator $T\in L(X)$ and its iterates, we can consider the following version of the Frequent Universality Criterion.

\begin{theorem}[{\bf Frequent Hypercyclicity Criterion}]
Let $X$ be a separable Fr\'{e}chet space and $T\in L(X)$. If there are a dense subset $X_0$\ of $X$ and $S_k:X_0\rightarrow X$, $k\ge 1$, such that for each $x\in X_0$,
\begin{enumerate}[\upshape 1.]
\item $\sum_{n=1}^{\infty}T^nx$ converges unconditionally,
\item $\sum_{n=1}^{\infty}S_nx$ converges unconditionally,
\item $T^mS_nx=S_{n-m}x$ for any $m< n$, and
\item $T^nS_nx=x$ for any $n\ge 1$.
\end{enumerate}
Then $T$ is frequently hypercyclic.
\end{theorem}

We remark that if $T$ satisfies the Frequent Hypercyclicity Criterion, then its sequence of iterates $(T^n)$ satisfies the Frequent Universality Criterion and there exists a dense subset $X_0$ of $X$ such that $T^{n}x \to 0$ for each $x\in X_0$. Therefore, we obtain the following version of the Criterion~$(M_k)$ for $\mathcal{U}$-frequently hypercyclic subspaces.

\begin{theorem}\label{UMkT}
Let $X$ be a separable infinite-dimensional Fr\'{e}chet space with a continuous norm and $T\in L(X)$ an operator satisfying the Frequent Hypercyclicity Criterion. 
Suppose that there exists a strictly increasing sequence of positive integers $(n_k)$ of positive upper density such that
$(T^{n_k})_{k\ge 1}$ is equicontinuous along some non-increasing sequence $(M_k)$ of infinite-dimensional closed subspaces of $X$. Then $T$ possesses a $\mathcal{U}$-frequently hypercyclic subspace.
\end{theorem}

In 2000,  Gonz\'{a}lez, Le\'{o}n and Montes~\cite{Gonzalez} obtained a characterization of operators with hypercyclic subspaces on complex Banach spaces. In particular, we can deduce from their proof that if $T$ is an operator on a complex Banach space satisfying the Hypercyclicity Criterion, then $T$ possesses a hypercyclic subspace if and only if $(T^k)_{k\ge 0}$ is equicontinuous along some non-increasing sequence $(M_k)$ of infinite-dimensional closed subspaces of $X$, see \cite[Theorem 3.2]{Gonzalez}. In view of Theorem~\ref{UMkT} and of the characterization of Gonz\'{a}lez, Le\'{o}n and Montes, we conclude the following.

\begin{theorem} \label{T:FHC2ufhcs}
Let $X$ be a separable infinite-dimensional complex Banach space and $T\in L(X)$ be an operator satisfying the Frequent Hypercyclicity Criterion. Then $T$ has a hypercyclic subspace if and only if it has a $\mathcal{U}$-frequently hypercyclic subspace.
\end{theorem}

Theorem~\ref{UMkT} also gives the following characterization of unilateral weighted shifts with $\mathcal{U}$-frequently hypercyclic subspaces on real or complex Banach spaces $\ell^p$  ($1\le p <\infty$).
\begin{cor}\label{caraclp}
Let $B_w:\ell^p\to \ell^p$ be a $\mathcal{U}$-frequently hypercyclic unilateral weighted shift.
Then $B_w$ has a $\mathcal{U}$-frequently hypercyclic subspace if and only if
 $B_w$ has a hypercyclic subspace.  Thus a unilateral backward shift $B_w$ with weight sequence $w=(w_n)$ has a $\mathcal{U}$-frequently hypercyclic subspace if and only if
$(\frac{1}{w_1\cdots w_n})\in\ell^p$ and $\sup_{n\ge 1} \inf_{k\ge 0} \prod_{\nu=1}^{n} |w_{k+\nu}|\le 1$.
\end{cor}
\begin{proof}
It is immediate that $B_w$ has a hypercyclic subspace if it has a $\mathcal{U}$-frequently hypercyclic subspace. Conversely, if $B_w$ has a hypercyclic subspace we can deduce from \cite{LeonMontes, Menet1} that there exists a non-increasing sequence of infinite-dimensional closed subspaces $(M_k)$ in $X$ such that  
\[\sup_{k\ge 1}\|B^k_{w|M_k}\|<\infty.\]
We also know thanks to Bayart and Ruzsa \cite{Bayart3} that if $B_w:\ell^p\to \ell^p$ is a $\mathcal{U}$-frequently hypercyclic weighted shift, then $B_w$ satisfies the Frequent Hypercyclicity Criterion. We therefore conclude by Theorem~\ref{UMkT} that $B_w$ has a $\mathcal{U}$-frequently hypercyclic subspace. The second assertion follows now from the facts that $B_w$ is $\mathcal{U}$-frequently hypercyclic on $\ell^p$ if and only if $(\frac{1}{w_1\cdots w_n})\in\ell^p$ \cite{Bayart3} and that $B_w$ has a hypercyclic subspace if and only if $\sup_{n\ge 1} \inf_{k\ge 0} \prod_{\nu=1}^{n} |w_{k+\nu}|\le 1$ \cite{LeonMontes, Menet1}.
\end{proof}
This result may seem surprising given that a weighted shift on $\ell^p$ is $\mathcal{U}$-frequently hypercyclic if and only if it is frequently hypercyclic \cite{Bayart3}, and there exist frequently hypercyclic weighted shifts with hypercyclic subspaces and no frequently hypercyclic subspace~\cite{Menet2}. In particular, we deduce the following.

\begin{cor} \label{C:ufhcsNofhcs}
There exists a  frequently hypercyclic operator that has $\mathcal{U}$-frequently hypercyclic subspaces but has no frequently hypercyclic subspace.
\end{cor}
The above results motivate the following.

\begin{problem}: Does there exist a $\mathcal{U}$-frequently hypercyclic operator possessing hypercyclic subspaces but no $\mathcal{U}$-frequently hypercyclic subspace? What if the operator is frequently hypercyclic?
\end{problem}

We complement our study of the existence of $\mathcal{U}$-frequently hypercyclic subspaces for weighted shifts by showing that every $\mathcal{U}$-frequently hypercyclic bilateral weighted shift on $\ell^p(\Z)$ possesses a frequently hypercyclic subspace and thus a $\mathcal{U}$-frequently hypercyclic subspace.

\begin{theorem}\label{caraclpbi}
Let $B_w$ be a bilateral weighted shift on $\ell^p(\Z)$. If $B_w$ is $\mathcal{U}$-frequently hypercyclic, then $B_w$ possesses a frequently hypercyclic subspace.
\end{theorem}
\begin{proof}
Let $B_w$ be a $\mathcal{U}$-frequently hypercyclic bilateral weighted shift on $\ell^p(\Z)$. We already know that $B_w$ satisfies the Frequent Hypercyclicity Criterion~\cite{Bayart3}. In view of Theorem~\ref{UM0}, it thus suffices to prove that there exists an infinite-dimensional closed subspace $M_0$ such that for any $x\in M_0$
\[B_w^{n}x\xrightarrow[n\to \infty]{} 0.\]
Since $B_w$ is $\mathcal{U}$-frequently hypercyclic, we also know that
$\sum_{n\le 0}|w_{0}\cdots w_n|^p<\infty$ ~\cite{Bayart3}. In particular, we have for any $k\ge 1$ 
\begin{equation}\label{wsUb}
\prod_{\nu=0}^{n}|w_{-k-\nu}|=\frac{\prod_{\nu=0}^{k+n}|w_{-\nu}|}{\prod_{\nu=0}^{k-1}|w_{-\nu}|}\xrightarrow[n\to \infty]{} 0.
\end{equation}
Let $k_0\ge 1$. We show that there exists $k_1\ge k_0$ such that for any $n\ge 0$,
\begin{equation*}
\prod_{\nu=0}^{n}|w_{-k_1-\nu}|\le 1.
\end{equation*}
Indeed, either for any $n\ge 0$, we have $\prod_{\nu=0}^{n}|w_{-k_0-\nu}|\le 1$ and we can consider $k_1=k_0$, or the set $F:=\{n\ge 0:\prod_{\nu=0}^{n}|w_{-k_0-\nu}|>1\}$ is non-empty.
We then remark that if $n\in F$, we have
\[\prod_{\nu=0}^{k_0+n}|w_{-\nu}|=\Big(\prod_{\nu=0}^{n}|w_{-k_0-\nu}|\Big)\Big(\prod_{\nu=0}^{k_0-1}|w_{-\nu}|\Big)>\prod_{\nu=0}^{k_0-1}|w_{-\nu}|.\]
We deduce from \eqref{wsUb} that $F$ has to be finite. Let $n_0:=\max F$ and $k_1=k_0+n_0+1$. We then have $k_1\ge k_0$ and for any $n\ge 0$,
\[\prod_{\nu=0}^{n}|w_{-k_1-\nu}|=\frac{\prod_{\nu=0}^{n_0+n+1}|w_{-k_0-\nu}|}{\prod_{\nu=0}^{n_0}|w_{-k_0-\nu}|}\le 1.\]
We conclude that there exists an increasing sequence of positive integers $(k_j)_{j\ge 0}$ such that for any $j\ge 0$, any $n\ge 0$, 
\begin{equation}\label{wsUb2}
\prod_{\nu=0}^{n}|w_{-k_j-\nu}|\le 1.
\end{equation}
Let $M_0:=\overline{\text{span}}\{e_{-k_j}:j\ge 0\}$ and $x\in M_0$. We have $x=\sum_{j=0}^{\infty}a_je_{-k_j}$ and for any $n\ge 1$, any $J\ge 0$, we deduce from \eqref{wsUb} and \eqref{wsUb2} that
\begin{align*}
\|B_w^nx\|^p&=\sum_{j=0}^{\infty}\Big(\prod_{\nu=0}^{n-1}|w_{-k_j-\nu}|\Big)^p |a_j|^p
\\
&\le \sum_{j=0}^J\Big(\prod_{\nu=0}^{n-1}|w_{-k_j-\nu}|\Big)^p |a_j|^p+ \sum_{j=J+1}^{\infty}|a_j|^p
\xrightarrow[n\to \infty]{}\sum_{j=J+1}^{\infty}|a_j|^p.
\end{align*}
We thus have $\|B_w^n x\|\to 0$ for any $x\in M_0$ and we conclude by using Theorem~\ref{UM0}.
\end{proof}

We finish this section by investigating the operators $P(D)$ on the space $H(\C)$ of entire functions where $P$ is a non-constant polynomial and $D$ is the derivative operator.
\begin{cor}\label{C:P(D)ufhs}
Let $D:H(\C)\to H(\C)$ be the derivative operator and $P$ a non-constant polynomial.
Then $P(D)$ has a $\mathcal{U}$-frequently hypercyclic subspace.
\end{cor}
\begin{proof}
We know that $P(D)$ satisfies the Frequently Hypercyclicity Criterion~\cite{Bonilla}. It is also shown in~\cite{Menet1} that if $P(D)$ satisfies the Hypercyclicity Criterion along a given strictly increasing sequence $(n_k)$, then there exists
 a non-increasing sequence of infinite-dimensional closed subspaces $(M_k)$ in $H(\C)$ such that for any $j\ge 1$, there exist a positive number $C_{j}$ and two integers $m(j),k(j)$
such that for any $k\ge k(j)$, any $x\in M_k$,
\[p_j(P(D)^{n_k}x)\le C_{j} p_{m(j)}(x).\]
In other words, it is shown that $(P(D)^{n_k})$ is equicontinuous along $(M_k)$.
Since $P(D)$ satisfies the Frequent Hypercyclicity Criterion, $P(D)$ satisfies the Hypercyclicity Criterion along the whole sequence $(k)$ and we can conclude by applying Theorem~\ref{UMkT} along the whole sequence $(n_k)=(k)$.
\end{proof}

Thanks to the following result by Delsarte and Lions, Corollary~\ref{C:P(D)ufhs} extends 
to
linear differential operators of finite order whose coefficients -except the leading one- may be non-constant.

\begin{lemma} {\rm (Delsarte and Lions \cite{Delsarte})} \label{L:Delsarte}
Let $T:H(\mathbb{C})\to H(\mathbb{C})$ be a differential operator of the form $T=D^N+a_{N-1}(z) D^{N-1}+\cdots + a_0(z) I$, where $N\ge 1$ and $a_j\in H(\mathbb{C})$ for $1\le j\le N$. Then there exists an onto isomorphism $U:H(\mathbb{C})\to H(\mathbb{C})$ so that $UT=D^NU$.
\end{lemma}

\begin{cor} \label{C:ufhsNonConstant}
For each $N\ge 1$ and $a_0,\dots, a_{N-1}\in H(\mathbb{C})$, the differential operator
\[
T=D^N+a_{N-1}(z) D^{N-1}+\cdots + a_0(z) I: H(\mathbb{C})\to H(\mathbb{C})
\]
has an $\mathcal{U}$-frequently hypercyclic subspace.
\end{cor}

\section{Existence of common hypercyclic subspaces}\label{Sec common}

\subsection{Criterion~$(M_k)$ for common hypercyclic subspaces}\label{Sec Com Mk}

In 2005, Bayart~\cite{Bayart} generalized Criterion~$M_0$ to the existence of common hypercyclic subspaces by using the Common Hypercyclicity Criterion of Costakis and Sambarino~\cite{Costakis2} and by using the approach introduced by Chan~\cite{Chan} to constructing hypercyclic subspaces via left-multiplication operators. With the same approach, Grosse-Erdmann and Peris showed that such Criterion~$M_0$ for common hypercyclic subspaces remains true with their own version of the Common Hypercyclicity Criterion~\cite{Grosse}. We note that a family of operators $\{ T_{\lambda}\}_{\lambda\in \Lambda}$ satisfying the Common Hypercyclicity Criterion of Costakis and Sambarino~\cite{Costakis2} also satisfies the Common Hypercyclicity Criterion given in~\cite{Grosse} and in particular it satifies $(CHC)$ given in Definition~\ref{D:CHC}. We give next a constructive proof of Criterion~$M_0$ for common hypercyclic subspaces for families of operators satisfying $(CHC)$.

\begin{theorem}[{\bf Criterion~${\bf M_0}$ for Common Hypercyclic Subspaces}]\label{CritM0}
Let $X$ be a Fr\'{e}chet space with a continuous norm, let $Y$ be a separable Fr\'{e}chet space and let $\{ (T_{n,\lambda})_{n\ge 0} \}_{\lambda\in \Lambda}$ be a family of sequences of operators in $L(X,Y)$ satisfying (CHC) and for which there exists an infinite-dimensional closed subspace $M_0$ such that for any $(\lambda, x) \in \Lambda\times M_0$, 
\[T_{n,\lambda}x\xrightarrow[n\to \infty]{} 0.\]
Then $\{ (T_{n,\lambda})_{n\ge 0} \}_{\lambda\in \Lambda}$ has a common hypercyclic subspace.

\end{theorem}
\begin{proof}
Let $(p_j)_{j\ge 1}$ be an increasing sequence of norms inducing the topology of~$X$, let $(q_j)_{j\ge 1}$ be an increasing sequence of seminorms inducing the topology of $Y$ and let $K=[a,b]\subset \Lambda$.
We denote by $X_{K,0}$, $Y_{K,0}$ and $S_{K,n,\lambda}$ the dense subsets and maps given by (CHC) for $K$.
We first show that for any $\varepsilon>0$, any $j\ge 1$, any $N_0\ge 0$, and any $y\in Y_{K,0}$, there exist      
 $N_1\ge N_0$  and $x\in X$ with $p_j(x)<\varepsilon$  so that for any $\lambda \in [a,b]$ there exists $k\in [N_0,N_1]$ satisfying
\begin{equation}
q_j(T_{k,\lambda}x-y)<\varepsilon.
\label{goodx}
\end{equation}

Let $\varepsilon>0$, $j\ge 1$, $N_0\ge 0$, and $y\in Y_{K,0}$ be given.
We consider $C\ge N_0$ such that
\begin{enumerate}
\item for any finite set $F\subset [C,\infty[$, any $l\ge 0$, any  $\lambda\ge \mu_0\ge \dots \ge \mu_l$ in $K$,
\[q_j\Big(\sum_{k\in F\cap [0,l]}T_{l,\lambda}S_{K,l-k,\mu_k}y\Big)<\varepsilon;\]
\item for any finite set $F\subset [C,\infty[$, any $l\ge 0$, any $\lambda\le \mu_0\le \mu_1\le \cdots$ in $K$, 
\[q_j\Big(\sum_{k\in F}T_{l,\lambda}S_{K,l+k,\mu_k}y\Big)<\varepsilon;\]
\item  for any finite set $F\subset [C,\infty[$, any $\mu_0\le \mu_1\le \dots$ in $K$.   
\[p_j\Big(\sum_{k\in F}S_{K,l,\mu_k}y\Big)<\varepsilon.\]
\end{enumerate}

Let $(\delta_{l})_{l\ge 0}$ be a sequence of positive real numbers such that ${\sum_{l=0}^{\infty}\delta_l=\infty}$ and for any $\alpha,\lambda \in K$, any $l\ge 0$,
\[0\le \alpha-\lambda\le \delta_l \quad\Rightarrow\quad q_j(T_{l,\lambda} S_{l,\alpha}y-y)<\varepsilon;\]
Since ${\sum_{l=0}^{\infty}\delta_l=\infty}$, there exists an increasing sequence $(k_l)_{l\ge 1}$ such that $\max\{C,N_0\}\le k_1$, for any $l\ge 1$, $k_{l+1}-k_l\ge C$ and $\sum_{l=1}^{\infty}\delta_{k_l}=\infty$. Such a sequence exists because
$(\{NC+k:N\ge 1\})_{k=0,\dots,N-1}$ forms a partition of $[C,\infty[$ and ${\sum_{l=C}^{\infty}\delta_l=\infty}$.

We select $L\ge 1$ the smallest integer such that 
\[\sum_{l=1}^{L}\delta_{k_l}\ge b-a.\]
Let $\lambda_0:=a$ and $\lambda_l:=\lambda_{l-1}+\delta_{k_l}$ for any $1\le l\le L$.
We deduce that $a=\lambda_0<\lambda_1<\cdots<\lambda_{L-1}\le b\le\lambda_L$.
We then consider	
\[x=\sum_{l=0}^{L-1} S_{K,k_{l+1},\lambda_l}y \quad\text{and}\quad N_1=k_L\]
and we show that $x$ and $N_1$ satisfy the desired properties. We first remark that $x$ satisfies
\[p_j(x)=p_j\Big(\sum_{l=0}^{L-1} S_{K,k_{l+1},\lambda_l}y\Big)<\varepsilon.\]
On the other hand, we see that for any $\lambda\in [a,b]$, there exists $0\le s\le L-1$ such that $\lambda\in [\lambda_s,\lambda_{s+1}]$ and thus
\begin{align*}
q_j(T_{k_{s+1},\lambda}x-y)&\le q_j(T_{k_{s+1},\lambda}S_{K,k_{s+1},\lambda_s}y-y)\\
&\quad +q_j\Big(\sum_{0\le l<s}T_{k_{s+1},\lambda}S_{K,k_{l+1},\lambda_l}y\Big)
+q_j\Big(\sum_{s<l\le L-1}T_{k_{s+1},\lambda}S_{K,k_{l+1},\lambda_l}y\Big)\\
&< q_j(T_{k_{s+1},\lambda}S_{K,k_{s+1},\lambda_s}y-y)+2\varepsilon\\
&\le 3\varepsilon \quad\quad\text{because }0\le \lambda-\lambda_s\le \lambda_{s+1}-\lambda_s\le \delta_{k_{s+1}}.
\end{align*}
Since for any $0\le s\le L-1$, $k_{s+1}\in[N_0,N_1]$, we conclude that \eqref{goodx} is satisfied.

Let $M_0$ be an infinite-dimensional closed subspace such that for any $\lambda\in \Lambda$, any $x\in M_0$, 
\[T_{k,\lambda}x\xrightarrow[k\to \infty]{} 0.\]
We consider a chain $(K_n)_{n\ge 1}$ of $\Lambda$ such that each $K_n$ is a compact subinterval, and
a basic sequence $(u_n)_{n\ge 1}$ in $M_0$ such that for every $n\ge 1$, we have $p_1(u_n)=1$ and the sequence $(u_k)_{k\ge n}$ is basic in $(X,p_n)$ with basic constant less than $2$. We remark that, since each $K_n$ is compact and $T_{l,\lambda}(x)$ depends continuously on $\lambda$, for any $n,k,j\ge 1$, there exist a positive number $C_{n,k,j}$ and a positive integer $m^0(n,k,j)$ such that for any $x\in X$,
\begin{equation}
\sup_{\lambda\in K_n}p_j(T_{k,\lambda}x)\le C_{n,k,j}p_{m^0(n,k,j)}(x). \quad\text{(Banach-Steinhaus)}
\label{cont Tlambda}
\end{equation}

In particular, if $\tilde{X}$ is a dense subset in $X$, we deduce from the previous reasoning that for any $\varepsilon>0$, any $j,n\ge 1$, any $N_0\ge 1$, any $y\in Y_{K_n,0}$, there exists $x\in \tilde{X}$ and $N_1\ge N_0$ such that $p_j(x)<\varepsilon$ and such that for any $\lambda \in K_n$, there exists $k\in [N_0,N_1]$ such that
\begin{equation}
q_j(T_{k,\lambda}x-y)<\varepsilon.
\label{superx}
\end{equation}

Let $(y_k)_{k\ge 1}$ be a dense sequence in $X_0$ and $\prec$ the order on $\mathbb{N}\times\mathbb{N}$ defined by $(i,j)\prec(i',j')$ if $i+j<i'+j'$ or if $i+j=i'+j'$ and $i<i'$. We construct a family $(z_{i,j})_{i,j\ge 1}\subset X$ and two families $(n^{0}_{i,j})_{i,j\ge 1}, (n^{1}_{i,j})_{i,j\ge 1} \subset \mathbb{N}$ such that for any $i,j\ge 1$, $n^{0}_{i,j}\le n^{1}_{i,j}$ and for any $i\ge 1$, $(n^{0}_{i,j})_{j\ge 1}, (n^{1}_{i,j})_{j\ge 1}$ are increasing. If $z_{i',j'}$, $n^{0}_{i',j'}$ and $n^{1}_{i',j'}$ are already constructed for every $(i',j')\prec(i,j)$, then we choose $z_{i,j}\in X$ and $n^0_{i,j}, n^1_{i,j}$ with $n^{0}_{i,j}>\max\{n^{1}_{i',j'}:(i',j')\prec(i,j)\}$ 
such that 
\begin{itemize}
\item we have, 
\begin{gather}
\sum_{j'\le j} z_{i,j'}\in X_{K_{i+j+1},0},
\label{X0}
\end{gather}
\item for any $n\ge n^{0}_{i,j}$, any $i'\ge 1$ we have
\begin{gather}
\sup_{\lambda\in K_{i+j}}q_{i+j}\Big(\sum_{j':(i',j')\prec (i,j)}T_{n,\lambda}z_{i',j'}\Big)<\frac{1}{2^{i'+j}}, \label{3n grand}
\end{gather}
\item we have
\begin{gather}
p_{i+j}(z_{i,j})<\frac{1}{2^{i+j+2}},\label{3z petit}
\end{gather}
\item for any $(i',j')\prec(i,j)$, any $\lambda\in K_{i+j}$, any $n\in[n^{0}_{i',j'},n^{1}_{i',j'}]$, we have
\begin{equation}
q_{i+j}(T_{n,\lambda}z_{i,j})<\frac{1}{2^{i+j+j'}},  \label{3T z petit}\\
\end{equation}
\item for any $\lambda\in K_{i+j}$, there exists $n\in[n^{0}_{i,j},n^{1}_{i,j}]$, such that we have
\begin{equation}
q_{i+j}(T_{n,\lambda}z_{i,j}-y_j)<\frac{1}{2^{j}}.
\label{close y}
\end{equation}
\end{itemize}
Satisfying \eqref{3n grand} is possible by choosing $n^{0}_{i,j}$ sufficiently big because for any $i'\ge 1$, if $i'<i$,
\[\sum_{j':(i',j')\prec (i,j)}z_{i',j'}=\sum_{j'\le i+j-i'}z_{i',j'}\in X_{K_{i+j+1},0}\]
and if $i'\ge i$,
\[\sum_{j':(i',j')\prec (i,j)}z_{i',j'}=\sum_{j'\le i+j-i'-1}z_{i',j'}\in X_{K_{i+j},0}.\]
Satisfying \eqref{3z petit} and \eqref{3T z petit} is possible by choosing $z_{i,j}$ sufficiently close to $0$ thanks to \eqref{cont Tlambda}, and we can choose $z_{i,j}$ and $n^{1}_{i,j}$ such that \eqref{X0} and \eqref{close y} are satisfied thanks to \eqref{superx}.

We define, for any $i\ge 1$,
\[z_i:=u_i+\sum_{j=1}^{\infty} z_{i,j}.\]
By \eqref{3z petit}, these series are convergent and we deduce that the sequence $(z_i)_{i\ge 1}$ is a basic sequence equivalent to $(u_i)_{i\ge 1}$ in $X$. Let $M$ be the closed linear span of $(z_i)$ and $z\in M\backslash\{0\}$. We need to show that $z$ is hypercyclic for each sequence $(T_{n,\lambda})$. Since $(z_i)$ is a basic sequence, we can write $z=\sum_{i=1}^{\infty}\alpha_i z_i$ for some scalar sequence $(\alpha_i)$, and re-scaling $z$ if necessary we can further assume that $\alpha_{k}=1$ for some $k$. By the equivalence between the basic sequences $(z_i)$  and $(u_i)$, we deduce that $\sum_{i=1}^{\infty}\alpha_i u_i$ also converges and that there exists $K>0$ such that for any $i\ge 1$ we have $|\alpha_i|\le K$.

Let $l\ge 1$, $\lambda\in \Lambda$ and $r\ge l$ be given. If $\lambda\in K_r$ we deduce by \eqref{3n grand}, \eqref{3T z petit} and \eqref{close y} that there exists $n\in[n^{0}_{k,r},n^{1}_{k,r}]$ such that
\begin{align*}
q_l(T_{n,\lambda}z-y_r)&\le q_l\Big(\sum_{(i,j)\prec(k,r)} \alpha_iT_{n,\lambda}z_{i,j}\Big) + q_l\Big(\sum_{(i,j)\succ(k,r)} \alpha_iT_{n,\lambda}z_{i,j}\Big)\\
&\quad + q_l(T_{n,\lambda}z_{k,r}-y_r) + q_l\Big(T_{n,\lambda}\Big(\sum_{i\ge1} \alpha_iu_i\Big)\Big)\\
&\le\sum_{i=1}^{\infty}K\Big(q_{k+r}\Big(\sum_{j:(i,j)\prec (k,r)}T_{n,\lambda}z_{i,j}\Big)+\sum_{j:(i,j)\succ(k,r)}q_{i+j}(T_{n,\lambda} z_{i,j})\Big)\\
&\quad+ q_{k+r}(T_{n,\lambda}z_{k,r}-y_r) + q_l\Big(T_{n,\lambda}\Big(\sum_{i\ge 1} \alpha_iu_i\Big)\Big)\\
&\le \sum_{i=1}^{\infty} K \Big(\frac{1}{2^{i+r}}+\sum_{j=1}^{\infty}\frac{1}{2^{i+j+r}}\Big) + \frac{1}{2^r} +q_l\Big(T_{n,\lambda}\Big(\sum_{i\ge 1} \alpha_iu_i\Big)\Big)\\
&\le \frac{2K+1}{2^{r}}+ q_l\Big(T_{n,\lambda}\Big(\sum_{i\ge 1} \alpha_iu_i\Big)\Big).
\end{align*}
The last quantity can be made arbitrarily small as long as $r$ is sufficiently large, since  $\sum_{i\ge 1} \alpha_iu_i\in M_0$. So $z$ is hypercyclic for $(T_{n,\lambda})_{n\ge 1}$.
\end{proof}

\begin{theorem}[{\bf Criterion~${\bf (M_k)}$ for Common Hypercyclic Subspaces}]\label{Mk com}
Let $X$ be a Fr\'{e}chet space with a continuous norm, let $Y$ be a separable Fr\'{e}chet space and let  $\{ (T_{n,\lambda})_{n\ge 0} \}_{\lambda\in \Lambda}$ a family of sequences in $L(X,Y)$ satisfying $(CHC)$. Suppose there exist a chain $(\Lambda_n)$ of $\Lambda$ and a non-increasing sequence $(M_k)$ of infinite-dimensional closed subspaces of $X$ 
so that for each $n\in\mathbb{N}$, 
\[
\{ (T_{k,\lambda})_{k\ge 1}\}_{\lambda\in\Lambda_n} \ \ \mbox{is uniformly equicontinuous along $(M_k)$.}
\]
Then $\{ (T_{n,\lambda})_{n\ge 0} \}_{\lambda\in \Lambda}$ has a common hypercyclic subspace.
\end{theorem}

For the proof of Theorem~\ref{UMk} we used a map $\phi$ to guarantee that the subsequence 
$\{ n_k:\ k\in I=\bigcup_{s\ge 1} [k_s,k_{s}+\phi(k_s)]$
given by Corollary~\ref{NiceMnone} kept the positiveness of the upper density of $(n_k)$. In the proof of Proposition~\ref{prophered} below we use the fact that given a divergent series $\sum_{k\ge 1} \delta_k $ of positive real numbers there exists a map  $\phi:\N\to\N$ so that for each increasing sequence of integers $(k_s)_{s\ge 1}$ in $\N$ the set $I=\bigcup_{s\ge 1} [k_s,k_{s}+\phi(k_s)]$ satisfies
\[\sum_{j\in I}\delta_{j}=\infty.\]
We prove Theorem~\ref{Mk com} after the proof of Proposition~\ref{prophered}.

\begin{prop}\label{prophered}
Let $X$ be a Fr\'{e}chet space supporting a continuous norm, let $Y$ be a separable Fr\'{e}chet space and let $\{ (T_{n,\lambda})_{n\ge 0 } \}_{\lambda\in \Lambda}$ be a family of sequences in $L(X,Y)$ satisfying the $(CHC)$. Then there exists a map $\phi:\N\to \N$ such that for any increasing sequence $(k_s)$ in $\mathbb{N}$, if $(m_k)=\bigcup_{s\ge 1}[k_s,k_s+\phi(k_s)]$, then $\{ (T_{m_k,\lambda})_{k\ge 1} \}_{\lambda\in \Lambda}$ satisfies $(CHC)$.
\end{prop}
\begin{proof}
Let $(q_j)_{j\ge 1}$ be an increasing sequence of seminorms inducing the topology of $Y$. Let $(K_n)_{n\ge 1}$ be a chain of $\Lambda$ such that each $K_n$ is a compact subinterval.  We denote by $X^{(n)}_{0}$, $Y^{(n)}_{0}$ and $S^{(n)}_{k,\lambda}$ the sets and maps given by $(CHC)$ for the compact set $K_n$. We remark that without loss of generality we can assume that for each $n\ge1$, $Y^{(n)}_{0}$ is countable. Moreover, for any $N\ge 1$, $j\ge 1$, $n\ge 1$, and $y_0\in Y^{(n)}_0$ we denote by $(\delta^{(n)}_{y_0,N,j,k})_k$ a sequence of positive real numbers such that $\sum_{k=0}^\infty \delta^{(n)}_{y_0,N,j,k} =\infty$ and for each $k\ge 0$ and $\lambda, \mu\in K_N$ we have
\[
0\le \mu-\lambda < \delta_{y_0,N,j,k}  \ \   \Rightarrow \ \ \  q_j( T_{k, \lambda} S^{(n)}_{k,\mu} y_0 - y_0 ) < \frac{1}{N}.
\]
We then choose $\phi$ such that for any $N\ge 1$, $j\ge 1$, $n\ge 1$, any $y_0\in Y^{(n)}_0$, if  $I=\bigcup_{s\ge 1}[k_s,k_s+\phi(k_s)]$, then 
\[\sum_{k\in I} \delta^{(n)}_{y_0,N,j,k} =\infty.\]
Such a map exists because, if $Y^{(n)}_0=(y_{n,l})_{l\ge 1}$, then it suffices to let $\phi(k)$ such that for any $n,j,l,N\le k$,
\[\sum_{i=k}^{k+\phi(k)} \delta^{(n)}_{y_{n,l},N,j,i}\ge 1.\]

Let $(k_s)$ be an increasing sequence and $(m_k)=\bigcup_{s\ge 1}[k_s,k_s+\phi(k_s)]$. We show that $\{(T_{m_k,\lambda})_{k\ge 1}\}_{\lambda\in \Lambda}$ satisfies $(CHC)$. We first remark that for any compact subset $K$ in $\Lambda$, there exists $n\ge 1$ such that $K_n\supset K$.
It thus suffices to prove that each condition of $(CHC)$ is satisfied for the compact subintervals $K_n$.
Let $x_0\in X^{(n)}_0$, $y_0\in Y_0^{(n)}$ and $q$ a continuous seminorm on $Y$.
\begin{enumerate}
\item[(1)]\   Assume that for any finite set $F\subset [C,\infty[$, any $l\ge 0$, any  $\lambda\ge \mu_0\ge \dots \ge \mu_l$ in $K_n$,
\[q\Big(\sum_{k\in F\cap [0,l]}T_{l,\lambda}S^{(n)}_{l-k,\mu_k}y_0\Big)<\varepsilon.\]
We remark that for any finite set $F\subset [C,\infty[$, any $l\ge C$, any  $\lambda\ge \mu_0\ge \dots \ge \mu_l$ in $K_n$, we have
\begin{align*}
q\Big(\sum_{k\in F\cap [0,l-1]}T_{m_l,\lambda}S^{(n)}_{m_{l-k},\mu_k}y_0\Big)
& = q\Big(\sum_{k'\in F'\cap [0,m_l]} T_{m_l,\lambda}S^{(n)}_{m_l-k',\nu_{k'}}y_0\Big),
\end{align*}
where $F'=\{k'\ge 0: k'= m_l-m_{l-k}, k\in F\cap [0,l-1]\}$ and for any $k'\in F'$, if $k'=m_l-m_{l-k}$, then $\nu_{k'}=\mu_{k}$. In particular, we have $F'\subset[m_l-m_{l-C},\infty[\subset [C,\infty[$ and thus we deduce that
\[
q\Big(\sum_{k\in F\cap [0,l-1]}T_{m_l,\lambda}S^{(n)}_{m_{l-k},\mu_k}y_0\Big)=
q\Big(\sum_{k'\in F'\cap [0,m_l]} T_{m_l,\lambda}S^{(n)}_{m_l-k',\nu_{k'}}y_0\Big)<\varepsilon.\]

\item[(2)]\ Assume that for any finite set $F\subset [C,\infty[$, any $l\ge 0$, any $\lambda\le \mu_0\le \mu_1\le \cdots$ in $K_n$, 
\[q\Big(\sum_{k\in F}T_{l,\lambda}S^{(n)}_{l+k,\mu_k}y_0\Big)<\varepsilon.\] 
We remark that for any finite set $F\subset [C,\infty[$, any $l\ge 1$, and any $\lambda\le \mu_0\le \mu_1\le \cdots$ in $K_n$, we have
\begin{align*}
q\Big(\sum_{k\in F}T_{m_l,\lambda}S^{(n)}_{m_{l+k},\mu_k}y_0\Big)
& = q\Big(\sum_{k'\in F'} T_{m_l,\lambda}S^{(n)}_{m_l+k',\nu_{k'}}y_0\Big),
\end{align*}
where $F'=\{k'\ge 0: k'= m_{l+k}-m_{l}, k\in F\}$ and for any $k'\in F'$, if $k'=m_{l+k}-m_{l}$, then $\nu_{k'}=\mu_{k}$. In particular, we have $F'\subset[m_{l+C}-m_{l},\infty[\subset [C,\infty[$ and thus we deduce that
\[
q\Big(\sum_{k\in F}T_{m_l,\lambda}S^{(n)}_{m_{l+k},\mu_k}y_0\Big)=
q\Big(\sum_{k'\in F'} T_{m_l,\lambda}S^{(n)}_{m_l+k',\nu_{k'}}y_0\Big)<\varepsilon.\]

\item[(3)]\  Let $\varepsilon>0$, $j\ge1$ and  $n\ge 1$. We consider $N\ge n$ such that $\frac{1}{N}<\varepsilon$. We know by choice of $\phi$ that $\sum_{k=1}^\infty \delta_{y_0,N,j,m_k} =\infty$, and we have, by definition of $\delta_{y_0,N,j,m_k}$, that for each $k\ge 1$ and $\lambda, \mu\in K_N$,
\[0\le \mu-\lambda < \delta_{y_0,N,j,m_k}  \ \   \Rightarrow \ \ \  q_j(T_{m_k, \lambda} S^{(n)}_{m_k,\mu} y_0 - y_0) < \epsilon.
\]
Since $K_N\supset K_n$, we have the desired result.
\end{enumerate}
 Conditions $(4)$ and $(5)$ are clear.
\end{proof}

\begin{proof}[Proof of Theorem~\ref{Mk com}]\
Let $(K_n)$ be a chain of $\Lambda$ such that each $K_n$ is a compact subinterval. 
We remark that the assumptions of Theorem~\ref{NiceMn} are satisfied if we consider $(\Lambda^0_n)=(\Lambda^1_n)=(K_n)$ and $(\Lambda^2_n)=(\Lambda_n)$. Indeed, we deduce that $(\Lambda^0_n)$ satisfies the required properties by using the Banach-Steinhaus theorem and the fact that $T_{n,\lambda}$ depends continuously on $\lambda$. 

On the other hand, as $\{ (T_{n,\lambda})_{n\ge 0} \}_{\lambda\in \Lambda}$ satisfies $(CHC)$, we deduce from Condition $(4)$ in Definition~\ref{D:CHC} that for any subinterval $[a,b]\subset \Lambda$, there exists a dense subset $X_{0}$ such that for any $x\in X_0$, $T_{k,\lambda}x\xrightarrow[k\to \infty]{}0$ uniformly on $\lambda\in [a,b]$. The chain $(\Lambda^1_n)$ thus satisfies the required properties.

We conclude by Proposition~\ref{prophered} and Theorem~\ref{NiceMn} that there exist an increasing sequence of positive integers $(m_k)$ and an infinite-dimensional closed subspace $M_0$ of $X$ 
 such that $(T_{m_k,\lambda})_{k\ge 1}$ satisfies $(CHC)$
and such that for each $x\in M_0$ and  $\lambda\in \Lambda$,
\[T_{m_k, \lambda}x\xrightarrow[k\to \infty]{} 0.\]
We obtain the desired result by applying Theorem~\ref{CritM0}.
\end{proof}

\subsection{Applications of Criterion $(M_k)$ for common hypercyclic subspaces}

In general, Criterion $(M_k)$ is much easier to use than Criterion~$(M_0)$ because in a lot of cases, when Criterion~$(M_k)$ is satisfied,  it is satisfied along the whole sequence $(k)$. We illustrate the use of Criterion $(M_k)$ for common hypercyclic subspaces on families of weighted shifts on K\"{o}the sequence spaces.

\begin{definition}
Let $A=(a_{j,k})_{j\ge 1,k\ge 0}$ be a matrix such that for any $j\ge 1$ and  $k\ge 0$, we have $a_{j,k}>0$ and $a_{j,k}\le a_{j+1,k}$.
The (real or complex) \emph{K\"{o}the sequence space} $\lambda^p(A)$  is defined as
\begin{align*}
\lambda^p(A)&:=\Big\{(x_k)_{k\ge 0}\in \omega : p_j((x_k)_k)=\Big(\sum_{k=0}^{\infty}|x_ka_{j,k}|^p\Big)^{\frac 1 p}<\infty, \ j\ge 1\Big\},
\end{align*}
endowed with the sequence of norms $(p_j)$.
\end{definition}

Let $\Lambda$ be an open interval of $\R$ and $w_{\lambda}$ be a sequence of non-zero scalars.
The weighted shift $B_{w_{\lambda}}$ is defined as $B_{w_{\lambda}}e_n=w_{\lambda,n}e_{n-1}$, where $e_{-1}=0$ and $(e_n)_{n\ge 0}$ is the canonical basis. We provide in Proposition~\ref{Bwlambda} a sufficient condition for a family of shifts
$\{ B_{w,\lambda} \}_{\lambda\in \Lambda}$  satisfying the $(CHC)$ on $\lambda^p(A)$ to support  a common hypercyclic subspace. We first note the following.
\begin{remark}\label{remequi}
Let $(p_j)_{j\ge 1}$ and $(q_j)_{j\ge 1}$ be increasing sequences of seminorms inducing the topologies of $X$, and $Y$, respectively.
If $T_{n,\lambda}$ depends continuously on $\lambda$ and $K$ is a compact subset of $\Lambda$, then $\{(T_{n,\lambda})_{n\ge 1}\}_{\lambda\in K}$ is uniformly equicontinuous along $(M_k)$ if and only if for any $j\ge 1$, there exists $C_j>0$ and $k(j),m(j)\ge 1$ such that for any $k\ge k(j)$, any $\lambda\in K$ and any $x\in M_k$,
\[q_j(T_{k,\lambda}x)\le C_j p_{m(j)}(x).\]
\end{remark}

\begin{prop} \label{Bwlambda}
Let $\lambda^p(A)$ be a K\"{o}the sequence space, $\Lambda$ an open interval of $\R$ and let $\{ B_{w_\lambda}\}_{\lambda\in \Lambda}$ be a family of weighted shifts on $\lambda^p(A)$ satisfying $(CHC)$. If there exist a chain of compact sets $(K_n)_{n\ge 1}$ of $\Lambda$, and for each $n\ge 1$ a sequence $(C_{n,j})_{j}$ of positive scalars  and a sequence $(m(n,j))_j$ of integers such that for any $l\ge 1$ and any $N\ge 0$,
\[\inf_{k\ge N}\max_{1\le n,j,m\le l}\sup_{\lambda\in K_n}\frac{p_j(B^m_{w_{\lambda}}e_k)}{C_{n,j}p_{m(n,j)}(e_k)}\le 1,\]
then $\{ B_{w_\lambda}\}_{\lambda\in \Lambda}$ has a common hypercyclic subspace.
\end{prop}
\begin{proof}
Let $(K_n)$, $(C_{n,j})$ and $(m(n,j))$ be sequences satisfying the above assumptions. For any $l\ge 1$, we consider $e_{n_l}$ such that $n_l>n_{l-1}$ and for any $n,j,m\le l$,
\[\sup_{\lambda\in K_n}\frac{p_j(B^m_{w_{\lambda}}e_{n_l})}{p_{m(n,j)}(e_ {n_l})}\le 2C_{n,j},\]

Let $M_k=\overline{\text{span}}\{e_{n_l}:l\ge k\}$.
We deduce that for any $n,j\ge 1$, any $\lambda\in K_n$, any $k\ge \max\{n,j\}$ and any $x\in M_k$ we have
\begin{align*}
p_j(B^{k}_{w_{\lambda}}x)^p&=
p_j\Big(B^{k}_{w_{\lambda}}\Big(\sum_{l=k}^{\infty}x_{n_l}e_{n_l}\Big)\Big)^p\\
&=\sum_{l=k}^{\infty}|x_{n_l}|^p p_j(B^{k}_{w_{\lambda}}e_{n_l})^p\\
&\le (2C_{n,j})^p \sum_{l=k}^{\infty}|x_{n_l}|^p p_{m(n,j)}(e_{n_l})^p\\
&\le (2C_{n,j})^p p_{m(n,j)}(x)^p.
\end{align*}
Since each $K_n$ is compact and $B_{w_{\lambda}}$ depends continuously on $\lambda$, we conclude by Remark~\ref{remequi} that each $\{ B_{w_{\lambda}}\}_{\lambda\in K_r}$ is uniformly equicontinuous along $(M_k)$. The conclusion now follows by Theorem~\ref{Mk com}.
\end{proof}

\begin{cor}\label{easycrit}
Let $\lambda^p(A)$ be a K\"{o}the sequence space. Let $\Lambda$ be an open interval of $\R$ and let $\{ B_{w_\lambda} \}_{\lambda\in \Lambda}$ be a family of weighted shifts on $\lambda^p(A)$ satisfying $(CHC)$. Suppose that for  each compact subset $K$ of $\Lambda$ there exist a sequence $(C_{K,j})$ of positive scalars and a sequence $(m(K,j))$ of positive integers such that for any $j\ge 1$ and  $n\ge 1$ we have
\[\limsup_{k\to \infty}\sup_{\lambda\in K}\frac{p_j(B^n_{w_{\lambda}}e_k)}{C_{K,j}\ p_{m(K,j)}(e_k)}\le 1.\]
Then $\{ B_{w_\lambda} \}_{\lambda\in \Lambda}$ has a common hypercyclic subspace.
\end{cor}

Corollary~\ref{easycrit} gives an answer to a question of Costakis and Sambarino~{\cite[Section 8]{Costakis2}}:

\begin{cor}\label{Ex:Costa}
For each $\lambda >1$, let $B_{w_\lambda}$ be the backward shift on $\ell^p$ $(1\le p<\infty)$ with weight sequence $w_{\lambda}=(1+\frac{\lambda}{k})_{k\ge 1}$.
Then $\{ B_{w_\lambda} \}_{\lambda >1}$ has a common hypercyclic subspace.
\end{cor}
\begin{proof}
Costakis and Sambarino \cite{Costakis2} have shown that the family
$\{ B_{w_{\lambda}}\}_{\lambda>1}$ satisfies (CHC).
But for each  $K=[a,b]\subset (0,\infty )$ and  $n\ge 1$ we have
\begin{align*}
\limsup_{k\to\infty}\sup_{\lambda\in K}\|B^n_{w_{\lambda}}e_k\|& = \limsup_{k\to\infty}\sup_{\lambda\in K} 
\Big(\prod_{\nu=0}^{n-1} |w_{\lambda,k-\nu}|\Big)\\
&= \limsup_{k\to\infty}\sup_{\lambda\in K}\Big(\prod_{\nu=0}^{n-1} \big(1+\frac{\lambda}{k-\nu}\big)\Big)\\
&\le \limsup_{k\to\infty}\ \big(1+\frac{b}{k-n+1}\big)^n=1.
\end{align*}
That is, the assumptions of Corollary~\ref{easycrit} are satisfied.
\end{proof}

Bayart and Ruzsa \cite{Bayart3} showed that a shift $B_w$ with weight sequence $w=(w_n)$ is frequently hypercyclic  on  $\ell^p$ if and only it is $\mathcal{U}$-frequently hypercyclic and if and only if $(\frac{1}{\prod_{1\le j\le n} w_j} )_{n\ge 1} \in \ell^p$.  
So each shift  $B_{w_{\lambda}}$ from Corollary~\ref{Ex:Costa} is frequently hypercyclic on $\ell^p$ for 	any $1\le p<\infty$.  Moreover, by  Corollary~\ref{caraclp}, each $B_{w_\lambda}$ has a $\mathcal{U}$-frequently hypercyclic subspace on $\ell^p$. This motivates the following.
\begin{problem}
For each $\lambda \in \C$, let $\{ B_{w_{\lambda}}\}_{\lambda>1}$ be the shift operator on $\ell^p$ with weight sequence $w_\lambda=(1+\frac{\lambda}{n})$. Does
$
\{ B_{w_\lambda} \}_{\lambda >1}
$
have a common $\mathcal{U}$-frequently hypercyclic subspace?
\end{problem}

We note that two operators having $\mathcal{U}$-frequently hypercyclic subspaces may fail to have a common hypercyclic subspace.

\begin{example} \label{E:T1T2}
Let $B_w$ be the backward shift operator on $\ell^2$ of weight sequence $w=(\frac{n+1}{n})$.  Then each of the operators $T_1=B_w\oplus 2B$ and $T_2=2B\oplus B_w$ on $X=\ell^2\oplus \ell^2$ has a $\mathcal{U}$-fhc subspace, but $\{ T_1, T_2 \}$ has no common hypercyclic subspace.  The latter follows from an argument used in \cite[Example 2.1]{Aron}:  For $j=1,2$, let $P_j:\ell^2\oplus\ell^2\to\ell^2$, $P_j(x_1, x_2)=x_j$ be the $j$-th orthogonal projection. If  $M\subset \ell^2\oplus\ell^2$ is a hypercyclic subspace for both $T_1$ and $T_2$, then $P_1(M)\cup P_2(M)$ consists (but zero) of common hypercyclic vectors for $B_w$ and $2B$. But since $M$ is closed and infinite-dimensional and $X$ is a Hilbert space, one of $P_1(M)$, $P_2(M)$ must {\em contain} a closed and infinite-dimensional subspace, contradicting the fact that $2B$ has no hypercyclic subspaces.  It remains to see that each of $T_1$ and $T_2$ has a $\mathcal{U}$-fhc subspace. Each of $T_1$, $T_2$ satisfies the FHC on $\ell^2\oplus\ell^2$, since each of their summands $B_w$ and $2B$ satisfies the FHC on $\ell^2$. Also, $B_w$ has a hypercyclic subspace on $\ell^2$, and 
there exists a non-increasing sequence $(M_k)$ of closed and infinite-dimensional subspaces  of $\ell^2$
for which  $\mbox{sup}_{k\ge 1} \| {B_w^k}_{|_{M_k}} \| <\infty$.  Thus $(T_1^k)$ and $(T_2^k)$ are equicontinuous along $(M_k\oplus\{ 0\})$ and $(\{ 0\}\oplus M_k)$, respectively, and the conclusion follows from Theorem~\ref{UMkT}
\end{example}



In 2010, Shkarin~\cite{ShkarinD} showed that the derivative operator $D$ on $H(\C)$ has a hypercyclic subspace.
By considering $H(\C)$ as the K\"{o}the sequence space $\Lambda^1(A)$ for $a_{j,k}=j^k$, we can now prove the existence of a common hypercyclic subspace for the family of operators $\{ \lambda D\}_{\lambda\in \C\backslash\{0\}}$ on $H(\C)$. 

\begin{cor}\label{mult D}
Let $D$ be the derivative operator on $H(\C)$.
Then the family $\{ \lambda D\}_{\lambda\in \C\backslash\{0\}}$ has a common hypercyclic subspace.
\end{cor}
\begin{proof}
We know that the family $\{ \lambda D\}_{\lambda> 0}$ satisfies $(CHC)$ \cite{Costakis2}. Moreover, for any $K=[a,b]\subset ]0,\infty[$ and any $n\ge 1$ we have
\begin{align*}
\limsup_{k\to\infty}\sup_{\lambda\in K}\frac{p_j(\lambda^nD^ne_k)}{p_{2j}(e_k)}
&=\limsup_{k\to\infty}\sup_{\lambda\in K}\frac{p_j(\lambda^nD^ne_{k+n})}{p_{2j}(e_{k+n})}\\
&= \limsup_{k\to\infty}\sup_{\lambda\in K}\frac{\lambda^n \prod_{\nu=1}^{n}(k+\nu)p_j(e_k)}{p_{2j}(e_{k+n})}\\
&\le \limsup_{k\to\infty}\frac{b^n (k+n)^n j^k}{(2j)^{k+n}}=0.
\end{align*}
Thus by Corollary~\ref{easycrit} the family $\{ \lambda D\}_{\lambda>0}$ has a common hypercyclic subspace. The conclusion now follows from the fact that for any complex scalar $\lambda$ and operator $T$ the operators
$\lambda T$ and $|\lambda| T$ have the same set of hypercyclic vectors \cite{Leon2}.
\end{proof}

Corollary~\ref{mult D} extends to operators $P(D)$ where $P$ is a non-constant polynomial.
\begin{prop} \label{P:P(D)common}
Let $D$ be the derivative operator on $H(\C)$ and let $P$ be a non-constant polynomial.
Then $\{ \lambda P(D)\}_{\lambda\in \C\backslash\{0\}}$ has a common hypercyclic subspace.
\end{prop}
\begin{proof}

 The family $\{ (\lambda^k P(D)^k)_{k\ge 1} \}_{\lambda>0}$ satisfies $(CHC)$ (see the  proof of \cite[Proposition 2.2]{Costakis} and Remark~\ref{R:CHC}).
On the other hand, we know~\cite{Menet1} that there exists a non-increasing sequence of infinite-dimensional closed subspaces $(M_k)$ in $H(\C)$ such that for any $j\ge 1$, there exist a positive number $C_{j}$ and two integers $m(j),k(j)$ such that for any $k\ge k(j)$, any $x\in M_k$,
\[p_j(P(D)^{k}x)\le C_{j} p_{m(j)}(x).\]
Without loss of generality, we can assume that for any $k\ge 1$, the subspace $M_k$ is included in $\overline{\text{span}}\{e_n:n\ge k\}$. We recall that the sequence $(p_j)$  of norms considered here is given by  $p_j((x_k)_k)=\sum_{k=0}^{\infty}|x_kj^k|$. In particular, we deduce from the previous inclusion that for any $x\in M_k$, any $j\ge 1$, any $n\ge 1$, \[p_j(x)\le n^{-k}p_{nj}(x).\]

Let $n\ge 1$. If we consider $\Lambda^{2}_n=[\frac{1}{n},n]$, $C_{n,j}=C_j$, $m(n,j)=nm(j)$, $k(n,j)=k(j)$, then for
any $\lambda\in \Lambda^2_n$, any $k\ge k(n,j)$, any $x\in M_k$, we have
\[p_j((\lambda P(D))^{k}x)\le \lambda^k C_{j} p_{m(j)}(x)\le \lambda^k C_{j} n^{-k}p_{nm(j)}(x)\le C_{j}p_{m(n,j)}(x).\]
By Remark~\ref{remequi} and Theorem~\ref{Mk com}, the family $\{ \lambda P(D)\}_{\lambda>0}$ has a common hypercyclic subspace. So $\{ \lambda P(D)\}_{\lambda\in \C\backslash\{0\}}$ has a common hypercyclic subspace, as for any operator $T$ and  any complex scalar $\lambda$ the operators $\lambda T$ and $|\lambda |T$ have the same set of hypercyclic vectors \cite{Leon2}.
\end{proof}

As with Corollary~\ref{C:ufhsNonConstant}, Lemma~\ref{L:Delsarte} by Delsarte and Lions allows to extend Proposition~\ref{P:P(D)common} to  any linear differential operators of finite order whose coefficients -other than the leading one- may be non-constant.

\begin{cor} \label{C:NonConstantCommon}
For each $N\ge 1$ and $a_0,\dots, a_{N-1}\in H(\mathbb{C})$, consider the differential operator
$
T=D^N+a_{N-1}(z) D^{N-1}+\cdots + a_0(z) I: H(\mathbb{C})\to H(\mathbb{C})
$.
Then
\[
\{ \lambda T \}_{\lambda >0}
\] has a common hypercyclic subspace.
\end{cor}

Recall that for any non-constant entire function  $\phi$ of exponential type the family  $\{  \lambda \phi (D) \}_{\lambda >0}$ has a common hypercyclic vector. This is due to Costakis and Mavroudis~\cite{Costakis} for the case $\phi$ is a polynomial, to Bernal  \cite{Bernal2009} when $\phi$ has exponential growth order no larger than $\frac{1}{2}$,  and due to Shkarin \cite{Shkarin} in the general case.
Proposition~\ref{P:P(D)common} motivates the following.
\begin{problem}
Let $\phi$ be entire, transcendental, and of exponential type.   Does the family  $\{ \lambda \phi (D) \}_{\lambda \in \C\setminus\{ 0 \} }$ have a common hypercyclic subspace? 
\end{problem}

\subsection{Multiples of an operator on a complex Banach space}\label{Sec:Charac}
A characterization by Gonz\'{a}lez, Le\'{o}n and Montes~\cite{Gonzalez} asserts that an operator $T$ on a complex Banach space satisfying the Hypercyclicity Criterion has a hypercyclic subspace if and only if its essential spectrum intersects the closed unit disc. We recall that if the essential spectrum of $T$ does not intersect the closed unit disk then for any infinite-dimensional closed subspace $M$, there exists $x\in M$ such that $\|T^n x\|\to \infty$. We denote by $\sigma_e(T)$ the essential spectrum of $T$, i.e. the set of complex scalars $\lambda$ such that $\lambda Id -T$ is not a Fredholm operator, where an operator is said to be Fredholm if its range is closed and its kernel and cokernel are finite-dimensional.

While it is likely not possible to extend such spectral characterization to arbitrary families of operators (even if they consist of two operators) as evidenced in \cite{Bes4}, we show that it indeed extends to certain 
 families of the form $\{ \lambda T\}_{\lambda\in\Lambda}$ thanks to Criterion~$(M_k)$ for common hypercyclic subspaces. 
Throughout this section, we let $X$ be a complex Banach space, $T\in L(X)$ and let $\mathcal{P}=(P_\lambda)_{\lambda\in \Lambda}$ be a family of polynomials. 
We denote by 
\[r_{\mathcal{P}}:=\inf\{r> 0: \exists \lambda\in \Lambda,\ P_\lambda(B(0,r)^c)\cap \overline{B(0,1)}= \emptyset\}.\]
We start by stating a sufficient condition for the non-existence of common hypercyclic subspaces for the family $\{ P_{\lambda}(T)\}_{\lambda\in\Lambda}$ in terms of the essential spectrum of~$T$.  

\begin{prop}\label{prop E}
If $\sigma_e(T)\cap \overline{B(0,r_{\mathcal{P}})}=\emptyset$ (with $\overline{B(0,0)}=\{0\}$), then
there exists $\lambda\in \Lambda$ such that each infinite-dimensional closed subspace $M$ of $X$ contains some $x$ for which
\[\|P^n_{\lambda}(T)x\|\xrightarrow[n\to \infty]{} \infty.\]
In particular, there exists $\lambda\in \Lambda$ such that $P_{\lambda}(T)$ has no hypercyclic subspace.
\end{prop}
\begin{proof}
Since $\sigma_e(T)$ is compact and $\sigma_e(T)\cap \overline{B(0,r_{\mathcal{P}})}=\emptyset$, there exists ${r>r_{\mathcal{P}}}$ such that 
\[\sigma_e(T)\cap B(0,r)=\emptyset.\]
By definition of $r_{\mathcal{P}}$, we know that there exists $\lambda\in \Lambda$ such that
 \[P_\lambda(B(0,r)^c)\cap \overline{B(0,1)}= \emptyset.\]
  Therefore, since $\sigma_e(P_\lambda(T))=P_\lambda(\sigma_e(T))$, we deduce that $\sigma_e(P_\lambda(T))\subset P_\lambda(B(0,r)^c)$ and thus
\[\sigma_e(P_\lambda(T))\cap \overline{B(0,1)}= \emptyset,\]
i.e. the essential spectrum of $P_{\lambda}(T)$ does not intersect the closed unit disk. 
We conclude that for any infinite-dimensional closed subspace $M\subset X$, there exists $x\in M$ such that
\[\|P^n_{\lambda}(T)x\|\xrightarrow[n\to \infty]{} \infty.\]
\end{proof}

Under certain conditions on the essential spectrum of~$T$, we can construct a non-increasing sequence $(M_n)$ of closed infinite-dimensional subspaces of $X$ so that the assumptions of Criterion~$(M_k)$ for common hypercyclic subspaces are met.

\begin{prop}\label{Mk}
 Suppose  $\sup_{\lambda\in \Lambda}|P_{\lambda}(\mu)|\le 1$ and that
 $\text{\emph{Ran}}(T-\mu Id)$ is dense in $X$, for some $\mu\in \sigma_{e}(T)$.
Then there exists a non-increasing sequence of infinite-dimensional closed subspaces $(M_n)$ in $X$ and a chain $(\Lambda_n)$ of $\Lambda$ such that for any $n\ge 1$, any $\lambda\in \Lambda_n$, any $m\le n$, we have
\[\|P^m_{\lambda}(T)x\|\le 2\|x\| \quad \text{for any $x\in M_n$.}\]
\end{prop}
\begin{proof}
Since $\text{Ran}(T-\mu Id)$ is dense and $\mu\in \sigma_{e}(T)$, either $\dim \ker(T-\mu Id)=\infty$ or $\text{Ran}(T-\mu Id)$ is not closed. If $\dim \ker(T-\mu Id)=\infty$, then we consider $M_n=\ker(T-\mu Id)$.  Therefore, for any $n\ge 1$, any $x\in M_n$, and any $\lambda\in \Lambda$ we have
\[\|P^n_{\lambda}(T)x\|=\|(P_{\lambda}(\mu))^n x\|\le \|x\|.\]

We now suppose that $\dim \ker(T-\mu Id)<\infty$ and $\text{Ran}(T-\mu Id)$ is not closed. We can then show that there exists an infinite-dimensional closed subspace $M$ in $X$ and a compact operator $K$ such that $T_{|M}=\mu Id + K$ (see \cite{Gonzalez}). Therefore, for any $n\ge 0$ we have $T^n_{|M}=\mu^n Id +K_n$ where $K_n$ is a compact operator. Since $K_n$ is compact, we know that for any $\varepsilon>0$, there exists a closed subspace of finite-codimension $E_{n,\varepsilon}$ in $X$ such that
$\|K_{n|E_{n,\varepsilon}}\|\le \varepsilon$.

We consider \[\Lambda_n=\{\lambda\in \Lambda: \text{deg}P_{\lambda}\le n \ \text{and}\ P_{\lambda}(x)=\sum_{k=0}^na_kx^k \ \text{with}\ \sum_{k=0}^{n}|a_k|\le n\}\] and we let $M_0:=M$ and 
\[M_n:=M_{n-1}\cap \bigcap_{k \le n^2}E_{k,\varepsilon_n} \quad\text{with}\quad \varepsilon_n=\frac{1}{n^{n}}.\]
We remark that for any $n\ge 1$, any $\lambda\in \Lambda_n$, any $m\ge 0$, 
\[P^m_\lambda(T)=\sum_{k=0}^{mn}c_kT^k \quad \text{with }\sum_{k=0}^{mn}|c_k|\le n^m.\]

We deduce that for any $n\ge 1$, any $\lambda\in \Lambda_n$, any $m\le n$, and any $x\in M_n$,
\begin{align*}
\|P^m_\lambda(T)x\|&= \Big\|\sum_{k=0}^{mn}c_kT^kx\Big\|=
\Big\|\sum_{k=0}^{mn}c_k(\mu^kx+K_kx)\Big\|\\
&\le |P_\lambda(\mu)|^m\cdot\|x\|+ \sum_{k=0}^{mn}(|c_k|\cdot \|K_{k|M_n}\|)\|x\|
\le \|x\| + n^{n} \varepsilon_n \|x\|\le 2\|x\|.
\end{align*}
\end{proof}

Thanks to Proposition~\ref{prop E} and Proposition~\ref{Mk}, we can now generalize the characterization of Gonz\'{a}lez, Le\'{o}n and Montes as follows:

\begin{theorem}\label{Charac com}
Let $X$ be a separable infinite-dimensional complex Banach space, $T\in L(X)$ and $\{P_{\lambda}\}_{\lambda\in \Lambda}$ a family of polynomials.
If $\{ P_{\lambda}(T)\}_{\lambda\in \Lambda}$ satisfies (CHC) and
if for any $\mu\in \overline{B(0,r_{\mathcal{P}})}$, $\text{\emph{Ran}}(T-\mu Id)$ is dense and ${\sup_{\lambda\in \Lambda}|P_{\lambda}(\mu)|\le 1}$, 
then the following assertions are equivalent:
\begin{enumerate}
\item for any $\lambda\in \Lambda$, $P_{\lambda}(T)$ has a hypercyclic subspace;
\item $\{ P_{\lambda}(T)\}_{\lambda\in \Lambda}$ has a common hypercyclic subspace;
\item $\{ P_{\lambda}(T))\}_{\lambda\in \Lambda}$ satisfies Criterion~$M_0$ for common hypercyclic subspaces;
\item $\{ P_{\lambda}(T)\}_{\lambda\in \Lambda}$ satisfies Criterion~$(M_k)$ for common hypercyclic subspaces;
\item the essential spectrum of $T$ intersects $\overline{B(0,r_{\mathcal{P}})}$.
\end{enumerate}
\end{theorem}
\begin{proof}
$\neg(5)\Rightarrow \neg (1)$ follows from Proposition~\ref{prop E}.

$(5)\Rightarrow (4)$. If the essential spectrum of $T$ intersects $\overline{B(0,r_{\mathcal{P}})}$, then by assumption, there exists $\nu\in \sigma_e(T)$ such that $\text{Ran}(T-\nu Id)$ is dense and ${\sup_{\lambda\in \Lambda}|P_{\lambda}(\nu)|\le 1}$. Since $(P_{\lambda}(T))_{\lambda\in \Lambda}$ satisfies (CHC), we deduce from Proposition~\ref{Mk} that
$(P_{\lambda}(T))_{\lambda\in \Lambda}$ satisfies the Criterion~$(M_k)$ for common hypercyclic subspaces.

$(4)\Rightarrow (3)$ and $(3)\Rightarrow (2)$ follow from Theorem~\ref{Mk com} and Theorem~\ref{CritM0}.

$(2)\Rightarrow (1)$ is immediate.
\end{proof}

In the case of scalar multiples of an operator, Theorem~\ref{Charac com} gives us the following two interesting characterizations.

\begin{cor} \label{C:a<lambda<b}
Let $T\in L(X)$, where  $X$ is a separable infinite-dimensional complex Banach space,
 and 
let $0< a<b$.
If $\{ \lambda T\}_{\lambda\in ]a,b[}$ satisfies (CHC),
then the following assertions are equivalent:
\begin{enumerate}
\item the family $\{ \lambda T\}_{\lambda\in ]a,b[}$ has a common hypercyclic subspace;
\item for any $\lambda\in ]a,b[$, the operator $\lambda T$ has a hypercyclic subspace;
\item the essential spectrum of $T$ intersects $\overline{B(0,\frac{1}{b})}$.
\end{enumerate}
\end{cor}
\begin{proof}
Let $\mathcal{P}=\{\lambda Id:\lambda\in ]a,b[\}$. We remark that $r_{\mathcal{P}}=\frac{1}{b}$ and
for any $\mu\in \overline{B(0,r_{\mathcal{P}})}=\overline{B(0,\frac{1}{b})}$,
we have 
\[\sup_{\lambda\in ]a,b[}|\lambda \mu|=b|\mu|\le 1.\]
Moreover, since $\{ \lambda T\}_{\lambda\in ]a,b[}$ satisfies (CHC), we know in particular that
$\lambda T$ is hypercyclic for any $\lambda\in ]a,b[$ and therefore that $\text{Ran}(T-\mu Id)$ is dense for any $\mu\in \C$~\cite{Bourdon}. The conclusion follows by Theorem~\ref{Charac com}.
\end{proof}

\begin{cor} \label{C:Spectral}
Let $T\in L(X)$, where
$X$ is a separable infinite-dimensional complex Banach space, and let  $a> 0$.
If $\{ \lambda T\}_{\lambda\in ]a,\infty[}$ satisfies (CHC),
then the following assertions are equivalent:
\begin{enumerate}
\item the family $\{ \lambda T\}_{\lambda\in ]a,\infty[}$ has a common hypercyclic subspace;
\item for any $\lambda> a$, the operator $\lambda T$ has a hypercyclic subspace;
\item $0\in \sigma_{e}(T)$.
\end{enumerate}
\end{cor}
\begin{proof}
If $\mathcal{P}=\{\lambda Id:\lambda\in ]a,\infty[\}$, we have $r_{\mathcal{P}}=0$ and thus $\overline{B(0,r_{\mathcal{P}})}=\{0\}$. We conclude as previously by using Theorem~\ref{Charac com}.
\end{proof}

We note that Corollary~\ref{C:Spectral} applies to several interesting examples by Gallardo and Partington \cite[Section 3]{Gallardo} of families of scalar multiples of either adjoint multiplication operators on the Hardy space, adjoint multiplier operators on weighted $\ell^2$-spaces, adjoint convolution operators on weighted $L^2(0,\infty)$ spaces, or adjoint composition operators on the reduced Hardy space. 
This is because any family $\{ \lambda T \}_{|\lambda | >a}$  satisfying the assumptions of  \cite[Theorem 2.1]{Gallardo} must satisfy (CHC).  We illustrate with the following.

\begin{example}
Suppose $\varphi\in H^\infty(\mathbb{D})$ is univalent and bounded below on $\mathbb{T}$ but is not an outer function. Suppose that zero belongs to the boundary of $\varphi (\mathbb{D})$.
Then 
\[
\{ \lambda M_\varphi^* \}_{ | \lambda | > a }
\]
has a common hypercyclic subspace on $H^2(\mathbb{D})$, where $a=\| \frac{1}{\varphi } {\|}_{ L_{ \mathbb{T} }^\infty }$ .
\end{example}
\begin{proof} We know that the spectrum of $M_\varphi^*$ is the closure of $\{ \overline{\varphi (z) }: \ z\in \mathbb{D} \}$ and that its essential spectrum is
$
\sigma_e(M_\varphi^*)=\partial \{ \overline{\varphi (z) }: \ z\in \mathbb{D} \}
$
thanks to $\varphi$ being univalent, see \cite[Proposition 4.4(b)]{Godefroy}. Hence the assumption
 $0\in \partial \varphi (\mathbb{D})$ gives that  $0\in \sigma_{e}(M_\varphi^*)$, and the conclusion follows by Corollary~\ref{C:Spectral}.
\end{proof}

\end{document}